\documentclass{article}
\usepackage{amssymb}
\usepackage{amsmath}
\usepackage{tikz}
\usepackage{graphicx}

\newcommand {\stirlingf}[2]{\genfrac[]{0pt}{}{#1}{#2}}
\newcommand {\stirlings}[2]{\genfrac\{\}{0pt}{}{#1}{#2}}
\newcommand {\lah}[2]{\genfrac\lfloor \rfloor{0pt}{}{#1}{#2}}

\newtheorem{theorem}{Theorem}

\newtheorem{corollary}[theorem]{Corollary}

\newtheorem{definition}[theorem]{Definition}

\newtheorem{lemma}[theorem]{Lemma}

\newtheorem{proposition}[theorem]{Proposition}
\newtheorem{remark}[theorem]{Remark}

\newenvironment{proof}[1][Proof]{\noindent\textbf{#1.} }{\ \rule{0.5em}{0.5em}}
\include {mak}
\begin{document}

\title{Associated Lah numbers and $r$-Stirling numbers}
\author{Hac\`{e}ne Belbachir and Imad Eddine Bousbaa \\
USTHB, Faculty of Mathematics\\
RECITS Laboratory, DG-RSDT\\
BP 32, El Alia, 16111, Bab Ezzouar, Algiers, Algeria\\
hbelbachir@usthb.dz \& ibousbaa@usthb.dz}
\maketitle
\date{}

\begin{abstract}
We introduce the associated Lah numbers. Some recurrence relations and
convolution identities are established. An extension of the associated
Stirling and Lah numbers to the $r$-Stirling and $r$-Lah numbers are also
given. For all these sequences we give combinatorial interpretation,
generating functions, recurrence relations, convolution identities. In the
sequel, we develop a section on nested sums related to binomial coefficient.
\end{abstract}



\section{Introduction}

The Stirling numbers of the first and second kind, denoted respectively $%
\stirlingf {n}{k}$ and $\stirlings {n}{k}$, are defined by%
\begin{equation}
x(x+1)\cdots (x+n-1)=\sum_{k=0}^{n}\stirlingf {n}{k}x^{k},
\label{ttt}
\end{equation}%
and%
\begin{equation}
x^{n}=\sum\limits_{k=0}^{n}\stirlings {n}{k}x(x-1)\cdots (x-k+1).
\label{ppp}
\end{equation}

It is well known that $\stirlingf {n}{k}$ is the number of permutations of
the set $Z_{n}:=\{1,2,\ldots ,n\}$ with $k$ cycles and that $\stirlings
{n}{k}$ is the number of partitions of the set $Z_{n}$ into $k$ non empty
subsets \cite[Ch. 5]{MR1999993}, \cite[Ch. 4]{MR1949650}.

The Lah numbers $\lah {n}{k}$ (Stirling numbers of the
third kind), see \cite[pp. 44]{MR1397498}, are defined as the sum of
products of the Stirling number of the first kind and the Stirling numbers
of the second kind%
\begin{equation}
\lah {n}{k}=\sum\limits_{j=k}^{n}\stirlingf {n}{j}%
\stirlings {j}{k},
\end{equation}%
and count the number of partitions of the set $Z_{n}$ into $k$ ordered
lists. According to \ref{ttt} and \ref{ppp}, they satisfy%
\begin{equation*}
x(x+1)\cdots (x+n-1)=\sum\limits_{k=0}^{n}\lah
{n}{k}x(x-1)\cdots (x-k+1),
\end{equation*}%
see for instance \cite[eq 8]{BelBel13}.

Broder \cite{MR743795} gives a generalization of the Stirling numbers of the
first and second kind the so-called $r$-Stirling numbers of the first and
second kind, denoted respectively $\stirlingf {n}{k}_{r}$ and $\stirlings
{n}{k}_{r}$, by adding restriction on the elements of $Z_{n}$ : the $\stirlingf {n}{k}_{r}$ is the number of permutations of the set $Z_{n}$ with $k$
cycles such that the $r$ first elements are in distinct cycles and the $%
\stirlingf {n}{k}_{r}$ is the number of partitions of the set $Z_{n}$ into $%
k $ subsets such that the $r$ first elements are in distinct subsets. The $r$%
-Lah numbers $\lah {n}{k}_{r}$, see \cite{BelBel13}, count
the number of partitions of the set $Z_{n}$ into $k$ ordered lists such that
the $r$ first elements are in distinct lists.

These three sequences satisfy respectively the following recurrence relations%
\begin{eqnarray}
\stirlingf {n}{k}_{r} &=&\stirlingf {n-1}{k-1}_{r}+\left( n-1\right) \stirlingf {n-1}{k}_{r},  \label{recS1} \\
\stirlings {n}{k}_{r} &=&\stirlings {n-1}{k-1}_{r}+k\stirlings{n-1}{k}_{r},  \label{recS22} \\
\lah {n}{k}_{r} &=&\lah{n-1}{k-1}_{r}+\left( n+k-1\right) \lah {n-1}{k}_{r}. \label{recLah}
\end{eqnarray}%
with $\stirlingf {n}{k}_{r}=\stirlings {n}{k}_{r}=\lah
{n}{k}_{r}=\delta _{n,k}$ for $k=r$ , where $\delta $ is the Kronecker
delta, and $\stirlingf {n}{k}_{r}=\stirlings {n}{k}_{r}=\lah {n}{k}_{r}=0$ for $n<r$.

For $r=1$ and $r=0$, these numbers coincide with the classical Stirling
numbers of both kinds and with the classical Lah numbers.

Comtet \cite[pp. 222]{MR1999993} define an other generalization of the
Stirling numbers of both kinds by adding a restriction on the number of
elements by cycle or subset and call them, for $s\geqslant 1$, the $s$%
-associated Stirling numbers of the first kind $\stirlingf {n}{k}^{(s)}$
and of the second kind $\stirlings {n}{k}^{(s)}$. The $\stirlingf {n}{k}%
^{(s)}$ is the number of permutations of the set $Z_{n}$ with $k$ cycles
such that, each cycle has at least $s$ elements. The $\stirlings
{n}{k}^{(s)}$ is the number of partitions of the set $Z_{n}$ into $k$
subsets such that, each subset has at least $s$ elements. They have, each
one, an explicit formula, see for instance \cite[Eq 4.2, Eq 4.9]{MR600368}:%
\begin{eqnarray}
\stirlingf {n}{k}^{(s)} &=&\frac{n!}{k!}\sum_{\substack{ i_{1}+i_{2}+\cdots
+i_{k}=n  \\ i_{j}\geq s}}\frac{1}{i_{1}i_{2}\cdots i_{k}},  \label{e} \\
\stirlings {n}{k}^{(s)} &=&\frac{n!}{k!}\sum_{\substack{ %
i_{1}+i_{2}+\cdots +i_{k}=n  \\ i_{j}\geq s}}\frac{1}{i_{1}!i_{2}!\cdots
i_{k}!}.  \label{ee}
\end{eqnarray}

The generating functions are respectively 
\begin{eqnarray}
\sum_{n\geq sk}\stirlingf {n}{k}^{(s)}\frac{x^{n}}{n!} &=&\frac{1}{k!}\left(
-\ln {(1-x)}-\sum_{i=1}^{s-1}\frac{x^{i}}{i}\right) ^{k},  \label{genes1} \\
\sum_{n\geq sk}\stirlings {n}{k}^{(s)}\frac{x^{n}}{n!} &=&\frac{1}{k!}%
\left( \exp \left( x\right) -\sum_{i=0}^{s-1}\frac{x^{i}}{i!}\right) ^{k}.
\label{genes2}
\end{eqnarray}

For $s=2$, these numbers are reduced to the specific associated Stirling
numbers of both kinds, see for instance \cite[pp. 73]{MR1949650}.

Note that, from (\ref{e}) and (\ref{ee}), for $n=sk$, we get 
\begin{equation}
\begin{tabular}{ll}
$\stirlingf {sk}{k}^{(s)}=\dfrac{\left( sk\right) !}{k!s^{k}}$ \ \ and & $%
\stirlings {sk}{k}^{(s)}=\dfrac{\left( sk\right) !}{k!\left( s!\right)^{k}}.$%
\end{tabular}
\label{koko}
\end{equation}

Ahuja and Enneking \cite{MR536963} give a generalization of the Lah numbers
called the associated Lah numbers using an analytic approach. In Sloane \cite%
[A076126]{MR1992789}, we have a definition of the associated Lah numbers $%
\lah {n}{k}^{(2)}$ as the number of partitions of the set $%
Z_{n}$ into $k$ ordered lists such that each list has at least $2$ elements.
They satisfy the following explicit formula%
\begin{equation}
\lah {n}{k}^{(2)}=\frac{n!}{k!}\binom{n-k-1}{k-1},
\end{equation}%
and have the double generating function 
\begin{equation}
\sum_{n\geqslant 2}\sum_{k=1}^{\left\lfloor n/2\right\rfloor
}\lah {n}{k}^{(2)}y^{k}\frac{x^{n}}{n!}=\exp \left( y\frac{x^{2}}{1-x}\right) -1,  \label{ddddddddddddddd}
\end{equation}%
they consider $k\geqslant 1$, which means there is at least one part.

Hsu and Shiue \cite{MR1618435} defined a Stirling-type pair $\{S^{1},S^{2}\}$
as a unified approach to the Stirling numbers, this approach generalize
degenerate Stirling numbers \cite{MR531621}, Weighted Stirling numbers \cite%
{MR570168, MR599657}, $r$-Whitney numbers \cite{MR1415279, MR2926106} and
many other ones. The authors and Belkhir in \cite{BelBou141} and the authors
in \cite{MR25} give a combinatorial approach to special cases of the
Stirling-type pair. Howard \cite{MR742846}\ extend the associated
generalization to the Weighted Stirling numbers. Note that the Stirling-type
pair does not generalize the associated Stirling numbers. Motivated by this,
we introduce and develop the $s$-associated Lah numbers and the $s$%
-associated $r$-Stirling numbers.

In section 2, we define the $s$-associated Lah numbers $\lah {n}{k}^{(s)}$, $n\geqslant sk$, by a combinatorial approach
analogous to Comtet's generalization. We derive an explicit formula, a
triangular recurrence relation, a combinatorial identity and some generating
functions. We study, in section 3, some nested sums related to binomial
coefficients in order to develop, in section 4, a generalization of the
Stirling numbers of the three kinds using the two restrictions (Broder's and
Comtet's ones), we call them respectively the $s$-associated $r$-Stirling
numbers of the first kind $\stirlingf {n}{k}_{r}^{(s)}$, the $s$-associated $%
r$-Stirling numbers of the second kind $\stirlings {n}{k}_{r}^{(s)}$ and
the $s$-associated $r$-Lah numbers $\lah {n}{k}_{r}^{(s)}$%
. We give some recurrence relations and combinatorial identities in sections
5 and 6. Cross recurrences and convolution identities are established in
sections 7 and 8. In section 9, we propose some generating functions of the $%
s$-associated $r$-Stirling numbers.

\section{The $s$-associated Lah numbers}

We start by introducing the $s$-associated Lah numbers.

\begin{definition}
The $s$-associated Lah number, denoted by $\lah
{n}{k}^{(s)}$, is the number of partitions of $Z_{n}$ into $k$ order lists
such that each list contains at least $s$ elements.
\end{definition}

\begin{theorem}
The $s$-associated Lah numbers obey to the following 'triangular' recurrence
relation, for $n\geqslant sk$,%
\begin{equation}
\lah {n}{k}^{(s)}=\binom{n-1}{s-1}s!\lah {n-s}{k-1}^{(s)}+(n+k-1)\lah {n-1}{k}^{(s)}, \label{srLah}
\end{equation}%
with $\lah {n}{0}^{(s)}=\delta _{n,0}$ for $k=0$, where $%
\delta $ is the Kronecker delta, and $\lah {n}{k}^{(s)}=0$
for $n<sk$
\end{theorem}

\begin{proof}
Let us consider the $n^{th}$ elements, if it belongs to a list containing
exactly $s$ elements, so we have $\binom{n-1}{s-1}$ ways to choose the
remaining $\left( s-1\right) $ elements and $s!$ ways to order them into the
cited list, then distribute the $\left( n-s\right) $ remaining elements into
the $\left( k-1\right) $ remaining lists such that each list have at least $%
s $ elements and we have $\lah {n-s}{k-1}^{(s)}$ ways to
do it. Thus, we get $\binom{n-1}{s-1}s!\lah
{n-s}{k-1}^{(s)}$ possibilities. Else, we consider all the possibilities of
ordering $\left( n-1\right) $ elements into $k$ lists under the usual
condition which can be done by $\lah {n-1}{k}^{(s)}$ ways,
then add the $n^{th}$ elements next to an other and we have $n-1$
possibilities, or as head of each list and we have $k$ possibilities, this
gives $(n+k-1)\lah {n-1}{k}^{(s)}$ possibilities.
\end{proof}

For $s=1$ and $s=2$, we get Lah numbers and associated Lah numbers
respectively.

For $s=3$, we obtain the following table, for $n\leq 15$,%
\begin{equation*}
\begin{tabular}{c|ccccc}
\hline
$n\backslash k$ & $1$ & $2$ & $3$ & $4$ & $5$ \\ \hline
$3$ & $6$ &  &  &  &  \\ 
$4$ & $24$ &  &  &  &  \\ 
$5$ & $120$ &  &  &  &  \\ 
$6$ & $720$ & $360$ &  &  &  \\ 
$7$ & $5040$ & $5040$ &  &  &  \\ 
$8$ & $40320$ & $60480$ &  &  &  \\ 
$9$ & $362880$ & $725760$ & $60480$ &  &  \\ 
$10$ & $3628800$ & $9072000$ & $1814400$ &  &  \\ 
$11$ & $39916800$ & $119750400$ & $39916800$ &  &  \\ 
$12$ & $479\,001\,600$ & $1676\,505\,600$ & $798\,336\,000$ & $19\,958\,400$
&  \\ 
$13$ & $6227\,020\,800$ & $\allowbreak 24\,\allowbreak 908\,083\,200$ & $%
15\,567\,552\,000$ & $1037\,836\,800$ &  \\ 
$14$ & $87\,\allowbreak 178\,291\,200$ & $392\,\allowbreak 302\,310\,400$ & $%
305\,124\,019\,200$ & $36\,\allowbreak 324\,288\,000$ &  \\ 
$15$ & $1307\,\allowbreak 674\,368\,000$ & $\allowbreak 6538\,\allowbreak
371\,840\,000$ & $6102\,480\,384\,000$ & $1089\,\allowbreak 728\,640\,000$ & 
$\allowbreak 10\,\allowbreak 897\,286\,400$ \\ \hline
\end{tabular}%
\end{equation*}

The following result gives an explicit formula for the $s$-associated Lah
numbers according to identities (\ref{e}) and (\ref{ee}) for the $s$%
-associated Stirling numbers of both kinds.

\begin{theorem}
Let $s,k$ and $n$ be nonnegative integers such that $n\geqslant sk$, we have%
\begin{equation}
\lah {n}{k}^{(s)}=\frac{n!}{k!}\binom{n-(s-1)k-1}{k-1}.
\label{lahr}
\end{equation}
\end{theorem}

\begin{proof}[Proof 1]
We order $n$ elements on $k$ ordered lists such that, each list contains at
least $s$ elements: first, we suppose that the lists are labeled $1,\ldots
,k $ and for each list $j$ we choose $(i_{j}+s)$ $(0\leqslant i_{j}\leqslant
n-s)$ elements, we have $\binom{n}{i_{1}+s,i_{2}+s,\ldots ,i_{k}+s}$
possibilities to constitute the $k$ groups. The arrangement of the $j^{th}$
subset gives $(i_{j}+s)!$ possibilities. It gets $\sum_{i_{1}+i_{2}+\cdots
+i_{k}=n-sk}\binom{n}{i_{1}+s,i_{2}+s,\ldots ,i_{k}+s}(i_{1}+s)!(i_{2}+s)!%
\cdots (i_{k}+s)!=n!\binom{n-(s-1)k-1}{k-1},$ we divide by $k!$ to unlabeled
the lists.
\end{proof}

\begin{proof}[Proof 2]
First we choose $k$ elements to identify the $k$ lists with $\binom{n}{k}$
possibilities, then we choose $k$ groups of $s-1$ elements to retch the
condition of having $s$ elements by list and we have $\binom{n-k}{s-1}$
possibilities for the first list, and $\binom{n-k-\left( s-1\right) }{s-1}$
possibilities for the second one, and so on ... the last list have $\binom{%
n-k-\left( s-1\right) (k-1)}{s-1}$ possibilities. So we get $\binom{n-k}{s-1}%
\binom{n-k-\left( s-1\right) }{s-1}\cdots \binom{n-k-\left( s-1\right) (k-1)%
}{s-1}=\binom{n-k}{s-1,s-1,\ldots ,s-1,n-sk}$ possibilities. We affect the
remaining $n-sk$ elements to the lists and we have $k$ ways for the first
element, $k+1$ ways for the second one, and so on \ldots\ the last one have $%
k+n-sk-1=n-(s-1)k-1$. So, we get $k\left( k+1\right) \cdots \left(
n-(s-1)k-1\right) =\frac{\left( n-(s-1)k-1\right) !}{\left( k-1\right) !}$
possibilities$.$ The result follows.
\end{proof}

\begin{remark}
For $n=sk$, we get the following according to relations given by (\ref{koko}%
) 
\begin{equation}
\lah {sk}{k}^{(s)}=\frac{\left( sk\right) !}{k!}.
\label{kokoko}
\end{equation}
\end{remark}

Comparing to (\ref{srLah}), an other recurrence relation, with rational
coefficients, can be deduced form the explicit formula \ref{lahr}, as
follows.

\begin{theorem}
The $s$-associated Lah numbers satisfy the following recurrence relation%
\begin{equation}
\lah {n}{k}^{(s)}=\frac{n!}{\left( n-s\right) !k}%
\lah {n-s}{k-1}^{(s)}+n\lah
{n-1}{k}^{(s)}.
\end{equation}
\end{theorem}

\begin{proof}
Using Pascal's formula and relation (\ref{lahr}), we get the result.
\end{proof}

Note that for $s=1$, we get the relation given by the authors \cite[eq 7]%
{BelBou14} when $r=0.$
.
The $s$-associated Lah numbers can be expressed as a Vandermonde type formula as follows.
\begin{figure}[!h]
\centering
\begin{tikzpicture}
	\draw (0,0)--(0,5)--(5,0);
	\draw[thin,loosely dotted] (0,0)--(2.5,0.0);
	\draw[thin,loosely dotted] (2.5,5)--(2.5,0.0);
	\draw[fill=black] (2.43,-0.07) rectangle(2.57,0.07);
	\node  at (2.5,5.3) {$k$};
	\node  at (-0.2,0.0) {$n$};
	\draw[thick] (2.5,0.6)--(1,3.6);
	\draw[thin,loosely dotted] (0,0.6)--(2.5,0.6);
	\draw[thin,loosely dotted] (2.5,5)--(2.5,0.6);
	\draw[fill=white] (2.43,0.53) rectangle(2.57,0.67);
	\node  at (-0.5,0.6) {$n-p$};
	
	\draw[thin,loosely dotted] (0,1.1)--(2.25,1.1) ; 
	\draw[thin,loosely dotted] (2.25,5)--(2.25,1.1);
	\draw[fill=white] (2.18,1.03) rectangle(2.32,1.17);
	\node  at (-1.2,1.1) {$n-p-(s-1)$};
	\draw[thin,loosely dotted] (0,1.6)--(2.0,1.6) ; 
	\draw[thin,loosely dotted] (2.0,5)--(2.0,1.6);
	\draw[fill=white] (1.93,1.53) rectangle(2.07,1.67);
	\node  at (-1.3,2.1) {$n-p-3(s-1)$};
	\draw[thin,loosely dotted] (0,2.1)--(1.75,2.1) ; 
	\draw[thin,loosely dotted] (1.75,5)--(1.75,2.1);
	\draw[fill=white] (1.68,2.03) rectangle(1.82,2.17);
	\node  at (-1.3,1.6) {$n-p-2(s-1)$};
	\draw[thin,loosely dotted] (0,3.6)--(1,3.6) ; 
	\draw[thin,loosely dotted] (1,5)--(1,3.6);
	\draw[fill=white] (0.93,3.53) rectangle(1.07,3.67);
	\node  at (1,5.3) {$k-p$};
	\node  at (-0.6,3.6) {$n-sp$};
\end{tikzpicture}
\caption{Value of $\lah{n}{k}^{(s)}$ (in black) as a inner product of a periodic sequence of elements of the same table (in white) with a sequence deriving from binomial coefficient.}
\label{Tux}
\end{figure}
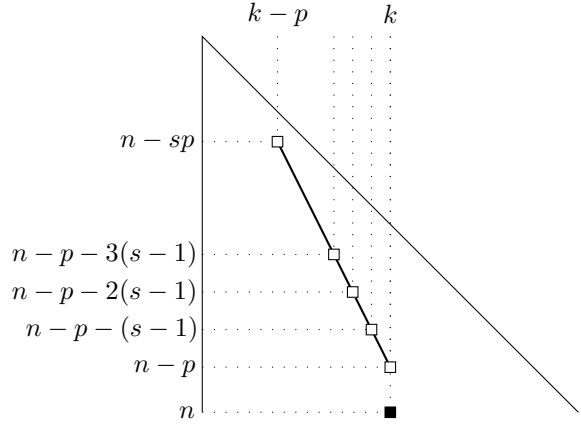
\begin{theorem}
The $s$-associated Lah numbers satisfy%
\begin{equation}
\lah {n}{k}^{(s)}=\frac{n!}{k!}\sum_{i=0}^{p}\frac{\left(
k-i\right) !}{\left( n-p-(s-1)i\right) !}\binom{p}{i}\lah
{n-(s-1)i-p}{k-i}^{(s)}.
\end{equation}
\end{theorem}

\begin{proof}
Using the explicit formula (\ref{lahr}) and the Vandermonde formula, we get
the result.
\end{proof}

The special case $s=1$ gives the identity given by the authors \cite[eq 6]%
{BelBou14} when $r=1$.

The $s$-associated Lah numbers satisfy the following vertical recurrence relation.
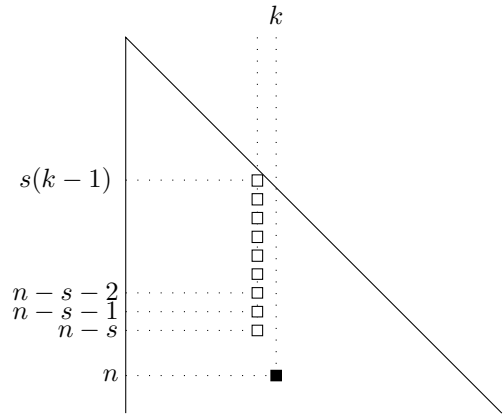
\begin{figure}[!h]
\centering
\begin{tikzpicture}
	\draw (0,0)--(0,5)--(5,0);
	\draw[thin,loosely dotted] (0,0.5)--(2.0,0.5);
	\draw[thin,loosely dotted] (2.0,5)--(2.0,0.5);
	\draw[fill=black] (1.93,0.43) rectangle(2.07,0.57);
	\node  at (2.0,5.3) {$k$};
	\node  at (-0.2,0.5) {$n$};
	\draw[thin,loosely dotted] (0,1.1)--(1.75,1.1) ; 
	\draw[thin,loosely dotted] (1.75,5)--(1.75,1.1);
	\draw[fill=white] (1.68,1.03) rectangle(1.82,1.17);
	\node  at (-0.5,1.1) {$n-s$};
	\draw[thin,loosely dotted] (0,1.35)--(1.75,1.35) ; 
	\draw[fill=white] (1.68,1.28) rectangle(1.82,1.42);
	\node  at (-0.8,1.35) {$n-s-1$};
	\draw[thin,loosely dotted] (0,1.6)--(1.75,1.6) ; 
	\draw[fill=white] (1.68,1.53) rectangle(1.82,1.67);
	\node  at (-0.8,1.6) {$n-s-2$};
	\draw[fill=white] (1.68,1.78) rectangle(1.82,1.92);
	\draw[fill=white] (1.68,2.02) rectangle(1.82,2.17);
	\draw[fill=white] (1.68,2.27) rectangle(1.82,2.42);
	\draw[fill=white] (1.68,2.52) rectangle(1.82,2.67);
	\draw[fill=white] (1.68,2.77) rectangle(1.82,2.92);
		\draw[thin,loosely dotted] (0,3.1)--(1.75,3.1) ;
	\draw[fill=white] (1.68,3.02) rectangle(1.82,3.17);

	\node  at (-0.8,3.1) {$s(k-1)$};
	
\end{tikzpicture}
\caption{linear vertical recurrence relation}
\label{Tu}
\end{figure}
\begin{theorem}
Let $s,k$ and $n$ be nonnegative integers such that $n\geqslant sk$, we have%
\begin{equation}
\lah {n}{k}^{(s)}=\sum_{i=s(k-1)}^{n-s}(n-i)!\binom{n-1}{i}%
\lah {i}{k-1}^{(s)}.  \label{RecVer}
\end{equation}
\end{theorem}

\begin{proof}
Let us consider the $\left( k-1\right) $ first lists, they contain $i$ $%
(s(k-1)\leqslant i\leqslant n-s)$ elements. So, we choose the $i$ elements
and we have $\binom{n-1}{i}$ ways to do it, and constitute the $k-1$ lists
such that each list have at least $s$ elements, which can be done by $%
\lah {i}{k-1}^{(s)}$ ways, then order the $(n-i)$
remaining elements in a list with $(n-i)!$ possibilities. We conclude by
summing over $i$.
\end{proof}

The exponential generating function of the $s$-associated Lah numbers is
given by the following. It is a complement list to (\ref{genes1}) and (\ref%
{genes2}).

\begin{theorem}
Let $n,k$ and $s$ be integers, we have%
\begin{equation}
\sum\limits_{n\geq sk}\lah {n}{k}^{(s)}\frac{x^{n}}{n!}=%
\frac{1}{k!}\left( \frac{x^{s}}{1-x}\right) ^{k}.  \label{generat}
\end{equation}
\end{theorem}

\begin{proof}
Using the explicit formula (\ref{lahr}), with the following identity for $%
x\in 
\mathbb{N}
$, see for instance \cite{MR0354401}, 
\begin{equation*}
\sum\limits_{n\geq 0}\binom{n+x}{x}t^{n}=\left( \frac{1}{1-t}\right) ^{x+1},
\end{equation*}%
we get the result.
\end{proof}

According to identity (\ref{ddddddddddddddd}), the double generating
function is given by

\begin{theorem}
We have%
\begin{equation}
\sum\limits_{n\geq sk}\sum\limits_{k=0}^{\left\lfloor n/s\right\rfloor
}\lah {n}{k}^{(s)}y^{k}\frac{x^{n}}{n!}=\exp \left\{ \frac{%
x^{s}}{1-x}y\right\} .
\end{equation}
\end{theorem}

\begin{proof}
Interchanging the order of summation and using equation (\ref{generat}), we
get the result.
\end{proof}

\section{Nested sums related to binomial coefficients}

In this section, we evaluate some symmetric functions. We start by the
following result.

\begin{lemma}
Let $\alpha $ and $\beta $ be integers such that\ $\beta \geqslant \alpha $.
We have%
\begin{equation}
\sum\limits_{n\geq 0}\binom{n+\alpha }{\beta }z^{n}=\frac{z^{\beta -\alpha }%
}{\left( 1-z\right) ^{\beta +1}}.  \label{y}
\end{equation}
\end{lemma}

\begin{proof}
$\sum\limits_{n\geq 0}\binom{n+\alpha }{\beta }z^{n}=\left(
\sum\limits_{n\geq 0}\binom{n+\beta }{\beta }z^{n}\right) z^{\beta -\alpha }=%
\frac{z^{\beta -\alpha }}{\left( 1-z\right) ^{\beta +1}}.$
\end{proof}

The following result seems to be nice as an independent one.

\begin{theorem}
Let $\alpha _{1},\ldots ,\alpha _{r},\alpha ,\beta _{1},\ldots ,\beta
_{r},\beta ,k_{1},\ldots ,k_{r}$ and $k$ be integers such that $\alpha
_{1}+\cdots +\alpha _{r}=\alpha $, $\beta _{1}+\cdots +\beta _{r}=\beta $
and $k_{1}+\cdots +k_{r}=k$ with $\beta _{i}\geqslant \alpha _{i}$. The
following identity holds%
\begin{equation}
\sum\limits_{k_{1}+\cdots +k_{r}=k}\binom{k_{1}+\alpha _{1}}{\beta _{1}}%
\cdots \binom{k_{r}+\alpha _{r}}{\beta _{r}}=\binom{k+\alpha +r-1}{\beta +r-1%
}.  \label{yy}
\end{equation}
\end{theorem}

\begin{proof}
By induction over $r$, we get the result. So It suffices to do the proof for 
$r=2.$ Thus, we have to establish 
\begin{equation}
\sum\limits_{k_{1}+k_{2}=k}\binom{k_{1}+\alpha _{1}}{\beta _{1}}\binom{%
k_{2}+\alpha _{2}}{\beta _{2}}=\binom{k+\alpha +1}{\beta +1}.  \label{dododo}
\end{equation}

We consider the following product $\sum\limits_{n\geq
0}\sum\limits_{k_{1}+k_{1}=k}\binom{k_{1}+\alpha _{1}}{\beta _{1}}\binom{%
k_{2}+\alpha _{2}}{\beta _{2}}t^{n}=\left( \sum\limits_{n\geq 0}\binom{%
k_{1}+\alpha _{1}}{\beta _{1}}t^{n}\right) \left( \sum\limits_{n\geq 0}%
\binom{k_{2}+\alpha _{2}}{\beta _{2}}t^{n}\right) $

using (\ref{y}), we get $\frac{t^{\beta _{1}-\alpha _{1}}}{\left( 1-t\right)
^{\beta _{1}+1}}\frac{t^{\beta _{2}-\alpha _{1}}}{\left( 1-t\right) ^{\beta
_{2}+1}}=\sum\limits_{k}\binom{k+\alpha +1}{\beta +1}t^{k}.$
\end{proof}

As a consequence, we evaluate the sum of all possible integer products
having the same summation.

\begin{corollary}
For $\alpha _{i}=0$ and $\beta _{i}=1$ we get%
\begin{equation}
\sum\limits_{k_{1}+\cdots +k_{r}=n}k_{1}k_{2}\cdots k_{r}=\binom{n+r-1}{2r-1}%
.  \label{abc}
\end{equation}
\end{corollary}

The above identity can be interpreted as the number of ways to choose $r$
leaders of $r$ groups constituted from $n$ persons: we choose one person of
each group and we have $\binom{k_{1}}{1}\cdots \binom{k_{r}}{1}$ ways to do
it. This is equivalent to choose $r$ persons and $\left( r-1\right) $
separators from the $n$ persons and the $r-1$ separators and we have $\binom{%
n+r-1}{2r-1}$ ways to do it.

Now, we are able to produce a general result. Also, it will be used to
establish the next theorem.

\begin{corollary}
Let $r,p$ and $k$ be integers such that $r\geqslant p$, we have%
\begin{equation}
\sum\limits_{k_{1}+\cdots +k_{p}+\cdots +k_{r}=n}k_{1}k_{2}\cdots k_{p}=%
\binom{n+r-1}{r+p-1}.  \label{dididid}
\end{equation}
\end{corollary}

\begin{proof}
\begin{equation*}
\sum\limits_{k_{1}+\cdots +k_{p}+\cdots +k_{r}=k}k_{1}k_{2}\cdots
k_{p}=\sum\limits_{m=0}^{k}\left( \sum\limits_{k_{1}+\cdots
+k_{p}=m}k_{1}k_{2}\cdots k_{p}\right) \sum\limits_{k_{p+1}+\cdots
+k_{r}=k-m}1
\end{equation*}%
using identity (\ref{abc}) and $\sum\limits_{i_{1}+i_{2}+\cdots +i_{r}=m}1=%
\binom{m+r-1}{r-1}$ we get%
\begin{equation*}
\sum\limits_{k_{1}+\cdots +k_{p}+\cdots +k_{r}=k}k_{1}k_{2}\cdots
k_{p}=\sum\limits_{m=0}^{k}\binom{m+p-1}{2p-1}\binom{k-m+r-p-1}{r-p-1},
\end{equation*}%
applying relation (\ref{dododo}), we get the result.
\end{proof}

Now, we are able to \ evaluate the sum of all products of $k$ terms, all
translated by $\alpha $, and having a fixed summation.

\begin{theorem}
We have%
\begin{equation}
\sum\limits_{i_{1}+\cdots +i_{k}=n}(i_{1}+\alpha )(i_{2}+\alpha )\cdots
(i_{k}+\alpha )=\sum\limits_{i=0}^{k}\dbinom{k}{j}\dbinom{n+k-1}{n-j}\alpha
^{k-j}.  \label{Getit}
\end{equation}
\end{theorem}

\begin{proof}
We have 
\begin{equation*}
\sum\limits_{i_{1}+\cdots +i_{k}=n}(i_{1}+\alpha )\cdots (i_{k}+\alpha
)=\sum\limits_{j=0}^{k}\dbinom{k}{j}I_{k,k-j}\alpha ^{j},
\end{equation*}%
where $I_{k,j}:=\sum\limits_{i_{1}+i_{2}+\cdots +i_{k}=n}i_{1}i_{2}\cdots
i_{j}$ and from (\ref{dididid}), we get the result.
\end{proof}

This nice result will be used to evaluate the explicit formula of the $s$%
-associated $r$-Stirling numbers which are introduced in the following
section.

\section{The $s$-associated $r$-Stirling numbers of the both kinds and the $s
$-associated $r$-Lah numbers}

Now, we introduce the $s$-associated $r$-Stirling numbers of the both kinds
and the $s$-associated $r$-Lah numbers.

\begin{definition}
The $s$-associated $r$-Stirling numbers of the first kind count the number
of permutations of the set $Z_{n}$ with $k$ cycles such that the $r$ first
elements are in distinct cycles and each cycle contains at least $s$
elements.

The $s$-associated $r$-Stirling numbers of the second kind count the number
of partitions of the set $Z_{n}$ into $k$ subsets such that the $r$ first
elements are in distinct subsets and each subset contains at least $s$
elements.

The $s$-associated $r$-Lah numbers, called also the $s$-associated $r$%
-Stirling numbers of the third kind, count the number of partitions of the
set $Z_{n}$ into $k$ ordered lists such that the $r$ first elements are in
distinct lists and each list contains at least $s$ elements.
\end{definition}

Here is given, for each kind, the table for $r=s=2$.

\begin{center}
\begin{equation*}
\begin{tabular}{c|cccccc}
\hline
$n\backslash k$ & $2$ & $3$ & $4$ & $5$ & $6$ & $7$ \\ \hline
$4$ & $2$ &  &  &  &  &  \\ 
$5$ & $12$ &  &  &  &  &  \\ 
$6$ & $72$ & $12$ &  &  &  &  \\ 
$7$ & $480$ & $160$ &  &  &  &  \\ 
$8$ & $3600$ & $1740$ & $90$ &  &  &  \\ 
$9$ & $30\,240$ & $18\,648$ & $2100$ &  &  &  \\ 
$10$ & $282\,240$ & $207\,648$ & $35\,840$ & $840$ &  &  \\ 
$11$ & $2903\,040$ & $2446\,848$ & $560\,448$ & $30\,240$ &  &  \\ 
$12$ & $32\,659\,200$ & $30\,702\,240$ & $8641\,080$ & $743\,400$ & $9450$ & 
\\ 
$13$ & $399\,168\,000$ & $410\,731\,200$ & $135\,519\,120$ & $15\,935\,920$
& $485\,100$ &  \\ 
$14$ & $5269\,017\,600$ & $5852\,753\,280$ & $2194\,121\,952$ & $%
324\,416\,400$ & $16\,216\,200$ & $124\,740$ \\ 
$15$ & $74\,724\,249\,600$ & $88\,663\,610\,880$ & $36\,941\,553\,792$ & $%
6522\,721\,920$ & $455\,975\,520$ & $8648\,640$ \\ 
$16$ & $1133\,317\,785\,600$ & $1424\,644\,865\,280$ & $649\,046\,990\,592$
& $132\,205\,465\,392$ & $11\,\allowbreak 835\,944\,120$ & $377\,116\,740$
\\ \hline
\end{tabular}%
\end{equation*}%
$\allowbreak \allowbreak $Table 1: Some values for the $2$-associated $2$%
-Stirling numbers of the first kind

\begin{equation*}
\begin{tabular}{c|ccccccc}
\hline
$n\backslash k$ & $2$ & $3$ & $4$ & $5$ & $6$ & $7$ & $8$ \\ \hline
$4$ & $2$ &  &  &  &  &  &  \\ 
$5$ & $6$ &  &  &  &  &  &  \\ 
$6$ & $14$ & $12$ &  &  &  &  &  \\ 
$7$ & $30$ & $80$ &  &  &  &  &  \\ 
$8$ & $62$ & $360$ & $90$ &  &  &  &  \\ 
$9$ & $126$ & $1372$ & $1050$ &  &  &  &  \\ 
$10$ & $254$ & $4788$ & $7700$ & $840$ &  &  &  \\ 
$11$ & $510$ & $15\,864$ & $45\,612$ & $15\,120$ &  &  &  \\ 
$12$ & $1022$ & $50\,880$ & $239\,190$ & $163\,800$ & $9450$ &  &  \\ 
$13$ & $2046$ & $159\,764$ & $1161\,270$ & $1389\,080$ & $242\,550$ &  &  \\ 
$14$ & $4094$ & $494\,604$ & $5353\,392$ & $10\,182\,480$ & $3638\,250$ & $%
124\,740$ &  \\ 
$15$ & $8190$ & $1516\,528$ & $23\,800\,920$ & $67\,822\,040$ & $%
41\,771\,730 $ & $4324\,320$ &  \\ 
$16$ & $16\,382$ & $4619\,160$ & $103\,096\,994$ & $422\,534\,112$ & $%
407\,246\,840$ & $85\,765\,680$ & $1891\,890$ \\ 
$17$ & $32\,766$ & $14\,004\,876$ & $438\,124\,050$ & $2507\,785\,280$ & $%
3555\,852\,300$ & $1280\,178\,900$ & $85\,135\,050$ \\ \hline
\end{tabular}%
\end{equation*}%
Table 2: Some values of the $2$-associated $2$-Stirling numbers of the
second kind

\begin{equation*}
\begin{tabular}{c|cccccc}
\hline
$n\backslash k$ & $2$ & $3$ & $4$ & $5$ & $6$ & $7$ \\ \hline
$4$ & $8$ &  &  &  &  &  \\ 
$5$ & $72$ &  &  &  &  &  \\ 
$6$ & $600$ & $\allowbreak 96$ &  &  &  &  \\ 
$7$ & $5280$ & $1920$ &  &  &  &  \\ 
$8$ & $50\,400$ & $29\,520$ & $\allowbreak 1440$ &  &  &  \\ 
$9$ & $524\,160$ & $428\,400$ & $50\,400$ &  &  &  \\ 
$10$ & $5927\,040$ & $6249\,600$ & $1229\,760$ & $\allowbreak 26\,880$ &  & 
\\ 
$11$ & $72\,576\,000$ & $93\,985\,920$ & $26\,490\,240$ & $1451\,520$ &  & 
\\ 
$12$ & $958\,003\,200$ & $1473\,292\,800$ & $546\,134\,400$ & $51\,408\,000$
& $\allowbreak 604\,800$ &  \\ 
$13$ & $13\,\allowbreak 571\,712\,000$ & $24\,\allowbreak 189\,580\,800$ & $%
11\,\allowbreak 176\,704\,000$ & $1536\,796\,800$ & $46\,569\,600$ &  \\ 
$14$ & $205\,\allowbreak 491\,686\,400$ & $416\,\allowbreak 731\,392\,000$ & 
$231\,\allowbreak 357\,772\,800$ & $42\,\allowbreak 471\,475\,200$ & $%
2255\,299\,200$ & $\allowbreak 15\,966\,720$ \\ 
$15$ & $3312\,\allowbreak 775\,065\,600$ & $7534\,\allowbreak 695\,168\,000$
& $4894\,\allowbreak 438\,348\,800$ & $1133\,\allowbreak 317\,785\,600$ & $%
89\,\allowbreak 253\,964\,800$ & $1660\,538\,880$ \\ \hline
\end{tabular}%
\end{equation*}%
Table 3: Some values for the $2$-associated $2$-Lah numbers
\end{center}

The $s$-associated $r$-Stirling numbers of the three kinds have the
following explicit formulas.

\begin{theorem}
For $n\geqslant sk$ and $k\geqslant r$, we have%
\begin{eqnarray}
\stirlingf {n}{k}_{r}^{(s)} &=&\frac{(n-r)!}{(k-r)!}\sum\limits_{m=0}^{n-sk}%
\binom{m+r-1}{r-1}\sum\limits_{i_{1}+\cdots +i_{k-r}=n-sk-m}\frac{1}{%
(i_{r+1}+s)\cdots (i_{k}+s)},  \label{EX1} \\
\stirlings {n}{k}_{r}^{(s)} &=&\frac{(n-r)!}{(k-r)!}\sum\limits_{\substack{
i_{1}+\cdots +i_{k}=n-sk}}\frac{1}{(i_{1}+s-1)!\cdots
(i_{r}+s-1)!(i_{r+1}+s)!\cdots (i_{k}+s)!},  \label{EXS2} \\
\lah {n}{k}_{r}^{(s)} &=&\frac{(n-r)!}{(k-r)!}%
\sum\limits_{j=0}^{r}\dbinom{r}{j}\binom{n-\left( s-1\right) k-1}{k+j-1}%
s^{r-j}.  \label{EXL}
\end{eqnarray}
\end{theorem}

\begin{proof}
We first proof the identity (\ref{EXS2}). To constitute a partition of $Z_{n}
$ into $k$ parts such that each part has at least $s$ elements and the $r$
first elements are in distinct parts, we proceed as follows : we put the $r$
first elements in $r$ parts (one by part). Now we partition the $n-r$
remaining elements into $k$ parts such that $r$ parts have at least $s-1$
elements and $k-r$ parts have at least $s$ elements, and we have $\frac{1}{%
(k-r)!}\binom{n-r}{i_{1},i_{2},\ldots ,i_{k}}$ ways to do it, with $%
i_{j}\geqslant s-1$ for $j=1,\ldots ,r$ and $i_{j}\geqslant s$ for $%
j=r+1,\ldots ,k$, which gives identity (\ref{EXS2}).

With the same specifications used to establish relation (\ref{EXS2}), to
count the number of permutations of $Z_{n}$ into $k$ cycles it suffice to
constitute the cycles by considering all the possible arrangement in the
parts and we have $i_{1}!i_{2}!\cdots i_{r}!(i_{r+1}-1)!\cdots (i_{k}-1)!$
ways. So we can write: 
\begin{eqnarray*}
\stirlingf {n}{k}_{r}^{(s)} &=&\frac{1}{(k-r)!}\sum\limits_{\substack{ %
i_{1}+\cdots +i_{k}=n-r}}\binom{n-r}{i_{1},\ldots ,i_{k}}i_{1}!\cdots
i_{r}!(i_{r+1}-1)!\cdots (i_{k}-1)!, \\
&=&\frac{(n-r)!}{(k-r)!}\sum\limits_{m=r(s-1)}^{n-r-s(k-r)}\sum\limits 
_{\substack{ i_{r+1}+\cdots +i_{k}=n-r-m  \\ i_{j}\geqslant s}}\frac{1}{%
i_{r+1}\cdots i_{k}}\sum\limits_{\substack{ i_{1}+\cdots +i_{r}=m  \\ %
i_{j}\geqslant s-1}}1, \\
&=&\frac{(n-r)!}{(k-r)!}\sum\limits_{m=0}^{n-sk}\sum\limits_{i_{r+1}+\cdots
+i_{k}=n-sk-m}\frac{1}{(i_{r+1}+s)\cdots (i_{k}+s)}\sum\limits_{i_{1}+\cdots
+i_{r}=m}1,
\end{eqnarray*}%
finally using $\sum\limits_{i_{1}+\cdots +i_{r}=m}1=\binom{m+r-1}{r-1}$, we
get identity (\ref{EX1}).

The same approach works, to constitute partitions of $Z_{n}$ into $k$
ordered lists we have to consider the arrangement in the parts and we have $%
\left( i_{1}+1\right) !\left( i_{2}+1\right) !\cdots \left( i_{r}+1\right)
!i_{r+1}!\cdots i_{k}!$ ways to do it. Thus we get%
\begin{eqnarray}
\lah {n}{k}_{r}^{(s)} &=&\frac{1}{(k-r)!}\sum\limits 
_{\substack{ i_{1}+\cdots +i_{k}=n-r}}\binom{n-r}{i_{1},\ldots ,i_{k}}\left(
i_{1}+1\right) !\cdots \left( i_{r}+1\right) !i_{r+1}!\cdots i_{k}!
\label{ih} \\
&=&\frac{(n-r)!}{(k-r)!}\sum\limits_{m=s(k-r)}^{n-sr}\sum\limits_{\substack{ %
i_{1}+\cdots +i_{r}=n-r-m  \\ i_{j}\geqslant s-1}}\left( i_{1}+1\right)
\cdots \left( i_{r}+1\right) \sum\limits_{\substack{ i_{r+1}+\cdots +i_{k}=m 
\\ i_{j}\geqslant s}}1  \notag \\
&=&\frac{(n-r)!}{(k-r)!}\sum\limits_{m=0}^{n-sk}\sum\limits_{\substack{ %
i_{1}+\cdots +i_{r}=n-sk-m  \\ i_{j}\geqslant 0}}\left( i_{1}+s\right)
\cdots \left( i_{r}+s\right) \sum\limits_{\substack{ i_{r+1}+\cdots +i_{k}=m 
\\ i_{j}\geqslant 0}}1,  \notag
\end{eqnarray}%
using relation (\ref{Getit}), we get%
\begin{equation*}
\lah {n}{k}_{r}^{(s)}=\frac{(n-r)!}{(k-r)!}%
\sum\limits_{j=0}^{r}\dbinom{r}{j}\sum\limits_{m=0}^{n-sk}\binom{m+k-r-1}{%
k-r-1}\dbinom{n-sk-m+r-1}{r-1+j}s^{r-j},
\end{equation*}%
using relation (\ref{dododo}), we get identity (\ref{EXL}).
\end{proof}

From (\ref{ih}), we can write a second kind explicit formula according to (%
\ref{Getit}) and generalizing relation (\ref{lahr}).%
\begin{equation}
\lah {n}{k}_{r}^{(s)}=\frac{(n-r)!}{(k-r)!}\sum\limits 
_{\substack{ i_{1}+i_{2}+\cdots +i_{k}=n-sk}}\left( i_{1}+s\right) \left(
i_{2}+s\right) \cdots \left( i_{r}+s\right) .  \label{32}
\end{equation}

The precedent theorem works for $k=r$. Furthermore, the identities are more
explicit.

\begin{remark}
For $k=r$, we get respectively%
\begin{eqnarray}
\stirlingf {n}{r}_{r}^{(s)} &=&\left( n-r\right) !\dbinom{n-r\left(
s-1\right) -1}{r-1},  \label{diga} \\
\stirlings {n}{r}_{r}^{(s)} &=&\sum\limits_{\substack{ i_{1}+i_{2}+\cdots
+i_{r}=n-r \\ i_{l}\geqslant s-1}}\binom{n-r}{i_{1},i_{2},\ldots ,i_{r}}, \\
\lah {n}{r}_{r}^{(s)} &=&\left( n-r\right)
!\sum\limits_{i=0}^{r}\dbinom{r}{i}\dbinom{n-(s-1)r-1}{r+i-1}s^{r-i}.
\end{eqnarray}
\end{remark}

The following special values can be easily computed, extending relations
given by relations (\ref{koko}) and (\ref{kokoko})%
\begin{eqnarray}
\stirlingf {sk}{k}_{r}^{(s)} &=&\frac{(n-r)!}{(k-r)!s^{k}}, \\
\stirlings {sk}{k}_{r}^{(s)} &=&\frac{(n-r)!}{(k-r)!(s-1)!^{r}s!^{k-r}}, \\
\lah {sk}{k}_{r}^{(s)} &=&\frac{(n-r)!}{(k-r)!s^{r}}.
\end{eqnarray}

Here is given an other explicit formula of the $s$-associated $r$-Lah
numbers. This one is more interesting than relation (\ref{32}). It is
evaluated using one summation

\begin{theorem}
Let $n,k,r$ and $s$ be nonnegative integers such that $k\geqslant r$ and $%
n\geqslant sk$, we have%
\begin{equation}
\lah {n}{k}_{r}^{(s)}=\frac{(n-r)!}{(k-r)!}%
\sum\limits_{j=0}^{r}\dbinom{r}{j}\dbinom{n+j-(s-1)k-1}{k+j-1}(s-1)^{r-j}.
\label{exp2}
\end{equation}
\end{theorem}

\begin{proof}
To constitute the $k$ lists we use the $r$ first elements which are supposed
in different lists to identify the $r$ first lists and we choose $k-r$
elements form the remaining elements, with $\binom{n-r}{k-r}$ possibilities,
as head list of the $k-r$ remaining lists. Now to retch the condition that
in each list we have at least $s$ elements, we constitute $k$ groups of $%
(s-1)$ elements form the $n-k$ remaining elements and we have $\binom{n-k}{%
\allowbreak s-1,\ldots ,s-1,n-sk}$ ways to do it, and consider all the
permutations of each group so we get $\left( (s-1)!\right) ^{k}$
possibilities. Now, for the $r$ first elements we suppose that $j$ of them
are head lists so we choose them with $\binom{r}{j}$ ways and order the $r-j$
elements after an element of each group and we have $(s-1)^{r-j}$
possibilities. It remains to affect the $n-sk$ remaining elements, so the
first one has $\left( k+j\right) $ possibilities ( $k:$ at the end of each
lists or before the $j$ supposed head lists), the second one have $\left(
k+j+1\right) $ possibilities (one possibilities added by the previews
element) and so one \ldots\ the last element have $\left( k+j\right) +\left(
n-sk-1\right) =n+j-(s-1)k-1$ possibilities. This gives $\frac{\left(
n+j-(s-1)k-1\right) !}{\left( k+j-1\right) !}=\left( k+j\right) \cdots
\left( n+j-(s-1)k-1\right) $ possibilities. Summing over all possible values
of $j$ we get$\binom{n-r}{k-r}\binom{n-k}{\allowbreak s-1,\ldots ,s-1,n-sk}%
\sum\limits_{j=0}^{r}\binom{r}{j}\frac{\left( n+j-(s-1)k-1\right) !}{\left(
k+j-1\right) !}$ $(s-1)^{r-j}$ which, after simplification, gives the result.
\end{proof}

Note that the explicit formula of the $s$-associated Lah numbers (\ref{lahr}%
) is obtained form (\ref{exp2}) for $r=0$ and $r=1$. Also, for $s=1$, we get
the explicit formula of the $r$-Lah numbers \cite[Eq 3]{BelBel13}.

From (\ref{exp2}) and (\ref{EXL}) we can state the following, which is very
nice in terms of identities related to binomial coefficients.

\begin{proposition}
We have 
\begin{equation}
\sum\limits_{j=0}^{r}\dbinom{r}{j}\dbinom{n+k+j-1}{n}(s-1)^{r-j}=\sum%
\limits_{j=0}^{r}\dbinom{r}{j}\binom{n-\left( s-1\right) k-1}{k+j-1}s^{r-j},
\label{cccc}
\end{equation}
\end{proposition}

From (\ref{ih}) and (\ref{exp2}) we get a second expression, dual to
relation (\ref{Getit}).

\begin{proposition}
we have%
\begin{equation*}
\sum\limits_{\substack{ i_{1}+i_{2}+\cdots +i_{k}=n}}\left( i_{1}+s\right)
\left( i_{2}+s\right) \cdots \left( i_{r}+s\right) =\sum\limits_{j=0}^{r}%
\dbinom{r}{j}\dbinom{n+k+j-1}{n}(s-1)^{r-j}
\end{equation*}
\end{proposition}

\section{Recurrence relations}

The $s$-associated $r$-Stirling numbers satisfy recurrence relations as the
regular $s$-associated Stirling numbers, using three terms of two triangles: the $(r-1)$-Stirling triangle and the $r$-Stirling triangle.

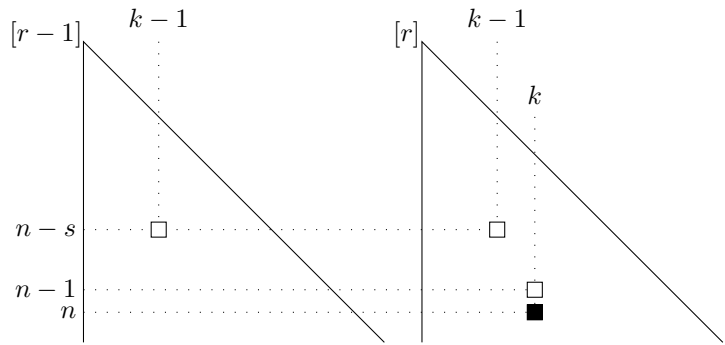
\begin{figure}[!h]
\centering
\begin{tikzpicture}
	\draw (0,0)--(0,4)--(4,0);

	\draw[thin,loosely dotted] (0,1.5)--(5.5,1.5);
	\draw[thin,loosely dotted] (1,4.0)--(1,1.5);
	\draw[fill=white] (0.9,1.4) rectangle(1.1,1.6);
	\node  at (1,4.3) {$k-1$};
	\node  at (-0.5,1.5) {$n-s$};
	\node  at (-0.5,0.7) {$n-1$};
	\draw (4.5,0)--(4.5,4)--(8.5,0);
	\draw[thin,loosely dotted] (0,0.4)--(6,0.4);
	\draw[thin,loosely dotted] (0,0.7)--(6,0.7);
	\draw[thin,loosely dotted] (5.5,4)--(5.5,1.5);
	\draw[thin,loosely dotted] (6.0,3)--(6.0,0.4);
	\draw[fill=white] (5.4,1.4) rectangle(5.6,1.6);
	\draw[fill=white] (5.9,0.6) rectangle(6.1,0.8);
	\draw[fill=black] (5.9,0.3) rectangle(6.1,0.5);
	\node  at (5.5,4.3) {$k-1$};
	\node  at (6,3.3) {$k$};
	\node  at (-0.2,0.4) {$n$};
	
	\node  at (4.3,4.1) {$[r]$};
	\node  at (-0.5,4.1) {$[r-1]$};

\end{tikzpicture}
\caption{Triangular recurrence relation given the value of the black element as linear combination of the values of the three others, for the $s$-associate $r$-Stirling numbers of the three kinds}
\label{Tulll}
\end{figure}
The recurrence relation of the $s$-associated $r$-Stirling numbers of the
first kind is given as follows.

\begin{theorem}
Let $r,k,s,$ and $n$ be nonnegative integers such that $n\geqslant sk$ and $%
k\geqslant r$, we have%
\begin{equation}
\stirlingf {n}{k}_{r}^{(s)}=\binom{n-r-1}{s-1}\left( s-1\right) !\stirlingf {%
n-s}{k-1}_{r}^{(s)}+r\binom{n-r-1}{s-2}\left( s-1\right) !\stirlingf {n-s}{%
k-1}_{r-1}^{(s)}+(n-1)\stirlingf {n-1}{k}_{r}^{(s)}.
\end{equation}
\end{theorem}

\begin{proof}
Let us consider the $n^{th}$ element, if it belongs to a cycle containing
exactly $s$ elements not from the $r$ elements, we have $\binom{n-r-1}{s-1}$
ways to choose the $\left( s-1\right) $ remaining elements and $\left(
s-1\right) !$ ways to constitute the cycle, then distribute the $\left(
n-s\right) $ remaining elements on the $\left( k-1\right) $ remaining cycles
such that each cycle has at least $s$ element and the $r$ first elements are
in distinct cycles, so we have $\stirlingf {n-s}{k-1}_{r}^{(s)}$ ways to do
it. Thus we get $\binom{n-r-1}{s-1}\left( s-1\right) !\stirlingf {n-s}{k-1}%
_{r}^{(s)}$ possibilities. Else, if one of the $r$ first elements belongs to
the cycle, we have $r$ ways to choose one of the $r$ first elements, $\binom{%
n-r-1}{s-2}$ ways to choose the remaining $\left( s-2\right) $ ones and $%
\left( s-1\right) !$ ways to constitute the cycle, then distribute the $%
\left( n-s\right) $ remaining elements on the $\left( k-1\right) $ remaining
cycles such that, in each cycle, there is at least $s$ elements and the $r-1$
first elements are in distinct cycles, so we have $\stirlingf {n-s}{k-1}%
_{r-1}^{(s)}$ possibilities to do it. Thus we get $r\binom{n-r-1}{s-2}\left(
s-1\right) !\stirlingf {n-s}{k-1}_{r-1}^{(s)}$ possibilities. Else, we
consider all the permutations of $\left( n-1\right) $ elements with $k$
cycles under the usual conditions which can be done by $\stirlingf {n-1}{k}%
_{r}^{(s)}$ ways, then add the $n^{th}$ element to the $k$ cycles and we
have $\left( n-1\right) $ possibilities.
\end{proof}

For $s=1$ we get relation (\ref{recS1}), and for $r=1$ using Pascal's
formula we get the recurrence relation of the $s$-associated Stirling
numbers of first kind \cite[eq 4.8]{MR600368}.

The $s$-associated $r$-Stirling numbers of the second kind satisfy the
following recurrence relation.

\begin{theorem}
Let $r,k,s,$ and $n$ be nonnegative integers such that $n\geqslant sk$ and $%
k\geqslant r$, we have%
\begin{equation}
\stirlings {n}{k}_{r}^{(s)}=\binom{n-r-1}{s-1}\stirlings
{n-s}{k-1}_{r}^{(s)}+r\binom{n-r-1}{s-2}\stirlings
{n-s}{k-1}_{r-1}^{(s)}+k\stirlings {n-1}{k}_{r}^{(s)}.
\end{equation}
\end{theorem}

\begin{proof}
Let us consider the $n^{th}$ elements, if it belongs to a part containing
exactly $s$ elements not from the $r$ first ones, so we have $\binom{n-r-1}{%
s-1}$ ways to choose the remaining $\left( s-1\right) $ elements and $%
\stirlings {n-s}{k-1}_{r}^{(s)}$ ways to distribute the $\left( n-s\right) 
$ remaining elements on the $\left( k-1\right) $ remaining parts such that,
the $r$ first elements are in distinct parts, and each part, have at least $%
s $ elements which gives $\binom{n-r-1}{s-1}\stirlings
{n-s}{k-1}_{r}^{(s)} $ possibilities. Else, if one of the $r$ first elements
belongs to that part, we have $r$ ways to choose it, and $\binom{n-r-1}{s-2}$
ways to choose the remaining $\left( s-2\right) $, then distribute the $%
\left( n-s\right) $ remaining elements on the $\left( k-1\right) $ remaining
parts such that, the $r-1$ first elements are in distinct parts, and each
part, have at least $s$ elements which can be done by $\stirlings
{n-s}{k-1}_{r-1}^{(s)}$ ways. So we have $r\binom{n-r-1}{s-2}\stirlings
{n-s}{k-1}_{r-1}^{(s)}$ possibilities. Else, we consider all the partitions
of $\left( n-1\right) $ elements on $k$ blocks under the usual conditions
which can be done by $\lah {n-1}{k}_{r}^{(s)}$ ways, then
add the $n^{th}$ element to the $k$ cycles with $\left( n-1\right) $
possibilities.
\end{proof}

For $s=1$ we get relation (\ref{recS22}), and for $r=1$ using Pascal's
formula we get the recurrence relation of the $s$-associated Stirling
numbers of the second kind \cite[eq 4.1]{MR600368}.

The $s$-associated $r$-Lah numbers satisfy the following recurrence relation.

\begin{theorem}
Let $r,k,s,$ and $n$ be nonnegative integers such that $n\geqslant sk$ and $%
k\geqslant r$ we have%
\begin{equation}
\lah {n}{k}_{r}^{(s)}=\binom{n-r-1}{s-1}s!\lah {n-s}{k-1}_{r}^{(s)}+r\binom{n-r-1}{s-2}s!\lah
{n-s}{k-1}_{r-1}^{(s)}+(n+k-1)\lah {n-1}{k}_{r}^{(s)}.
\end{equation}
\end{theorem}

\begin{proof}
Let us consider the $n^{th}$ element, if it belongs to a list containing
exactly $s$ elements not from the $r$ first ones, we have $\binom{n-1}{s-1}$
ways to choose the remaining $\left( s-1\right) $ elements and $s!$ ways to
constitute the list, then distribute the $\left( n-s\right) $ remaining
elements into the $\left( k-1\right) $ remaining lists such that each list
has at least $s$ elements and the $r$ first elements are in distinct lists
with $\lah {n-s}{k-1}_{r}^{(s)}$ ways. Thus we get $\binom{%
n-1}{s-1}s!\lah {n-s}{k-1}^{(s)}$ possibilities. Else, if
one of the $r$ first elements belongs to the list, we have $\binom{r}{1}=r$
ways to choose one of the $r$ first elements and $\binom{n-r-1}{s-2}$ ways
to choose the remaining $\left( s-2\right) $ elements and $s!$ ways to
constitute the list, then distribute the $\left( n-s\right) $ remaining
elements into the $\left( k-1\right) $ remaining lists such that each list
has at least $s$ elements and the $r-1$ first elements are in distinct lists
and we have $\lah {n-s}{k-1}_{r-1}^{(s)}$ ways to do it.
Thus we get $r\binom{n-r-1}{s-2}s!\lah
{n-s}{k-1}_{r-1}^{(s)}$ possibilities. Else, we consider all the partitions
of $\left( n-1\right) $ elements into $k$ lists under the usual conditions
which can be done by $\lah {n-1}{k}_{r}^{(s)}$ ways, then
add the $n^{th}$ element to the $k$ lists and we have $\left( n-1\right) $
possibilities after each element or $k$ possibilities as a head list, which
gives $(n+k-1)\lah {n-1}{k}_{r}^{(s)}$ possibilities.
\end{proof}

For $s=1$ we get relation (\ref{recLah}), and for $r=1$ and using Pascal's
formula we get the recurrence relation (\ref{srLah}).

\section{Combinatorial identities or convolution relations}

In this section, we establish some combinatorial identities for the $s$%
-associated $r$-Stirling numbers using a combinatorial approach. we can also consider them as convolution relations.

The next identity is an expressions of $s$-associated $r$-Stirling numbers
in terms of the $s$-associated $r^{\prime }$-Stirling numbers with $%
r^{\prime }\leqslant r$.

\begin{theorem}
Let $p,r,k$ and $n$ be nonnegative integers such that $p\leqslant r\leqslant
k$ and $n\geqslant sk$, we have%
\begin{equation}
\stirlingf {n}{k}_{r}^{(s)}=\sum\limits_{i=(s-1)p}^{n-p-s\left( k-p\right) }%
\frac{\left( n-r\right) !}{\left( n-r-i\right) !}\dbinom{i-p\left(
s-2\right) -1}{p-1}\stirlingf {n-p-i}{k-p}_{r-p}^{(s)}.  \label{S1rsfiS1p}
\end{equation}
\end{theorem}

\begin{proof}
Let us consider the $i$ $\left( (s-1)p\leqslant i\leqslant n-p-s\left(
k-p\right) \right) $ elements which belongs to the $p$ cycles containing the
elements $1,\ldots ,p$. We have $\dbinom{n-r}{i}$ possibilities to choose
the $i$ elements and $\stirlingf {i+p}{p}_{p}^{(s)}$ ways to construct the
corresponding cycles. The remaining $n-p-i$ elements must form the $k-p$
remaining cycles; this can be done in $\stirlingf {n-p-i}{k-p}_{r-p}^{(s)}$
ways. Using equation (\ref{diga}) and summing for all $i$, we get the proof.
\end{proof}

For $p=r$, we obtain an expression of the $s$-associated $r$-Stirling
numbers of the first kind in terms of the regular $s$-associated Stirling
numbers of the first kind%
\begin{equation}
\stirlingf {n}{k}_{r}^{(s)}=\sum\limits_{i=(s-1)r}^{n-r-s(k-r)}\frac{\left(
n-r\right) !}{\left( n-r-i\right) !}\dbinom{i-r\left( s-2\right) -1}{r-1}%
\stirlingf {n-r-i}{k-r}^{(s)}.  \label{S1rsS1s}
\end{equation}

For $s=1$, we obtain the equation given by Broder \cite[eq 26]{MR743795} and
for $r=1$, we get a vertical recurrence relation for the classical $s$%
-associated Stirling numbers of the first kind%
\begin{equation}
\stirlingf {n}{k}^{(s)}=\sum\limits_{i=s-1}^{n-s(k-1)-1}\frac{\left(
n-1\right) !}{\left( n-i-1\right) !}\stirlingf {n-i-1}{k-1}^{(s)}.
\end{equation}

\begin{theorem}
Let $p,r,k$ and $n$ be nonnegative integers such that $p\leqslant r\leqslant
k$ and $n\geq sk$, we have%
\begin{equation}
\stirlings {n}{k}_{r}^{(s)}=\sum\limits_{i=p-r+s(k-p)}^{n-r-\left(
s-1\right) p}\frac{\left( n-r\right) !}{\left( \left( s-1\right) !\right)
^{p}\left( n-p(s-1)-r\right) !}\dbinom{n-p(s-1)-r}{i}\stirlings
{i+r-p}{k-p}_{r-p}^{(s)}p^{n-p(s-1)-r-i}.
\end{equation}
\end{theorem}

\begin{proof}
Let us consider $p$ first elements $(p\leqslant r)$, they constitute $p$
parts with $p(s-1)$ elements so we choose those elements by $\binom{n-r}{%
s-1,\ldots ,s-1,n-sp(s-1)-r}=\frac{\left( n-r\right) !}{\left( \left(
s-1\right) !\right) ^{p}\left( n-p(s-1)-r\right) !}$ ways. Then we choose
the $i$ elements $\left( (s-1)(r-p)+s(k-r)\leqslant i\leqslant n-r-\left(
s-1\right) p\right) $ which belongs to the remaining $k-p$ parts and we have 
$\binom{n-p(s-1)-r}{i}$ ways to do it. Then, distribute them on $k-p$ parts
such that the $r-p$ fixed elements are in distinct parts and each part have
at least $s$ elements, which can be done by $\stirlings
{i+r-p}{k-p}_{r-p}^{(s)}$ possibilities. It remains now to distribute the
remaining $n-p(s-1)-r-i$ elements on the $p$ first parts and we have $%
p^{n-p(s-1)-r-i}$ possibilities. We conclude by summing over all possible
values of $i$.
\end{proof}

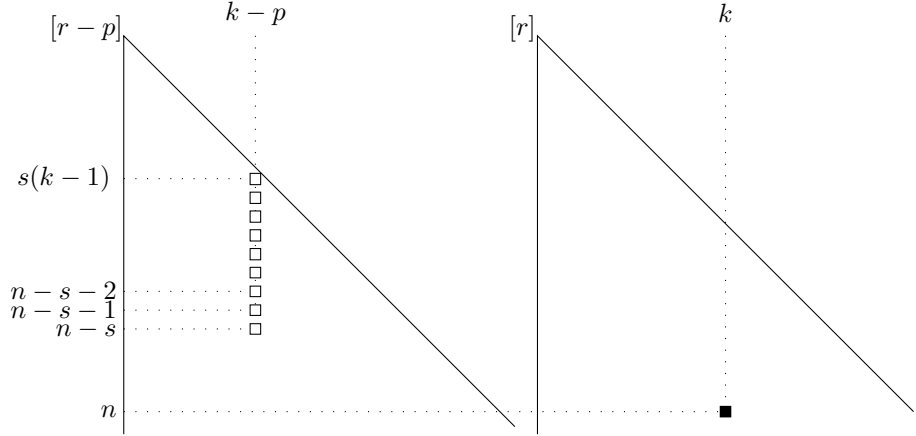
\begin{figure}[!h]
\centering
\begin{tikzpicture}
	\draw (0,-0.3)--(0,5)--(5.2,-0.2);
	\draw[thin,loosely dotted] (0,1.1)--(1.75,1.1) ; 
	\draw[thin,loosely dotted] (1.75,5)--(1.75,1.1);
	\draw[fill=white] (1.68,1.03) rectangle(1.82,1.17);
	\node  at (-0.5,1.1) {$n-s$};
	\node  at (1.75,5.3) {$k-p$};
	\draw[thin,loosely dotted] (0,1.35)--(1.75,1.35) ; 
	\draw[fill=white] (1.68,1.28) rectangle(1.82,1.42);
	\node  at (-0.8,1.35) {$n-s-1$};
	\draw[thin,loosely dotted] (0,1.6)--(1.75,1.6) ; 
	\draw[fill=white] (1.68,1.53) rectangle(1.82,1.67);
	\node  at (-0.8,1.6) {$n-s-2$};
	\draw[fill=white] (1.68,1.78) rectangle(1.82,1.92);
	\draw[fill=white] (1.68,2.02) rectangle(1.82,2.17);
	\draw[fill=white] (1.68,2.27) rectangle(1.82,2.42);
	\draw[fill=white] (1.68,2.52) rectangle(1.82,2.67);
	\draw[fill=white] (1.68,2.77) rectangle(1.82,2.92);
		\draw[thin,loosely dotted] (0,3.1)--(1.75,3.1) ;
	\draw[fill=white] (1.68,3.02) rectangle(1.82,3.17);

	\node  at (-0.8,3.1) {$s(k-1)$};
	
	\draw (5.5,-0.3)--(5.5,5)--(10.5,0);
	\draw[thin,loosely dotted] (0,0)--(8,0.0);
	\draw[thin,loosely dotted] (8,5)--(8,0.0);
	\draw[fill=black] (7.93,-0.07) rectangle(8.07,0.07);
	\node  at (8,5.3) {$k$};
	\node  at (-0.2,0.0) {$n$};
	
	\node  at (5.3,5.1) {$[r]$};
	\node  at (-0.5,5.1) {$[r-p]$};
	
\end{tikzpicture}
\caption{The value of an element in the $s$-associated $r$-Stirling table in terms of the consecutive vertical elements in the $s$-associated $r-p$-Stirling table as an inner product result}
\label{Tuu}
\end{figure}

For $p=r$ we get an expression of the $s$-associated $r$-Stirling numbers of
the second kind in terms of the regular $s$-associated Stirling numbers of
the second kind%
\begin{equation}
\stirlings {n}{k}_{r}^{(s)}=\sum\limits_{i=s(k-r)}^{n-sr}\frac{\left(
n-r\right) !}{\left( \left( s-1\right) !\right) ^{r}\left( n-sr\right) !}%
\dbinom{n-sr}{i}\stirlings {i}{k-r}^{(s)}r^{n-sr-i},
\end{equation}%
also, for $s=1$, we obtain the equation given by Broder \cite[eq 31]%
{MR743795} and for $r=1$, we get a vertical recurrence relation for the
classical $s$-associated Stirling numbers of the second kind

\begin{equation}
\stirlings {n}{k}^{(s)}=\sum\limits_{i=s(k-1)}^{n-s}\dbinom{n-1}{s-1}%
\dbinom{n-s}{i}\stirlings {i}{k-1}^{(s)}.
\end{equation}

\begin{theorem}
Let $p,r,k$ and $n$ be nonnegative integers such that $p\leqslant r\leqslant
k$ and $n\geq sk$, we have%
\begin{equation}
\stirlings {n}{k}_{r}^{(s)}=\sum_{i=0}^{p}\frac{(k-r+p-i)!}{(k-r)!}\binom{%
p}{i}\binom{n-r}{i(s-1)}\frac{\left( i(s-1)\right) !}{\left( \left(
s-1\right) !\right) ^{i}}\stirlings {n-p-i(s-1)}{k-i}_{r-p}^{(s)}.
\label{hoba}
\end{equation}
\end{theorem}

\begin{proof}
Let us consider the $p$ first elements, and focus on the $i$ $(0\leq i\leq
p) $ parts containing exactly $s$ elements, we have $\binom{p}{i}$ ways to
choose the $i$ elements from the $p$ first ones, and $\binom{n-r}{i(s-1)}$
ways to choose the $i(s-1)$ remaining elements to have $s$ elements by part,
and $\stirlings {\left( i(s-1)\right) }{i}^{(s-1)}=\frac{\left(
i(s-1)\right) !}{i!\left( \left( s-1\right) !\right) ^{i}}$ ( from \ref{koko}%
) ways to partition the $i(s-1)$ elements on $i$ groups such that each group
have at least $(s-1)$ elements, then affect each group to the $i$ elements
and we have $i!$. Then, we partition the $n-p-i(s-1)$ remaining elements
into $\left( k-i\right) $ parts such that each group has at least $s$
elements and the remaining $r-p$ elements are in distinct subsets, and we
have $\stirlings {n-r-i(s-1)}{k-i}_{r-p}^{(s)}$ ways to do it. Now, it
reminds $\left( p-i\right) $ elements not yet affected. Thus we have $\left(
k-r+p-i\right) $ choice for the first one, $\left( k-r+p-i-1\right) $ choice
for the second one and so on until the last one have $(k-r+1)$ which gives $%
\left( k-r+p-i\right) \left( k-r+p-i-1\right) \cdots (k-r+1)=\frac{(k-r+p-i)!%
}{(k-r)!}$ possibilities. We conclude by summing.
\end{proof}

For $s=1$ we get the relation given by the authors \cite[eq 5]{BelBou12}.

An expression of the $s$-associated $r$-Stirling numbers of the second kind
in terms of the $s$-associated Stirling numbers can be deduced from equation
(\ref{hoba}), for $p=r$, as follows

\begin{equation}
\stirlings {n}{k}_{r}^{(s)}=\sum_{i=0}^{r}\frac{(k-i)!}{(k-r)!}\binom{r}{i%
}\binom{n-r}{i(s-1)}\frac{\left( i(s-1)\right) !}{\left( \left( s-1\right)
!\right) ^{i}}\stirlings {n-r-i(s-1)}{k-i}^{(s)}.  \label{gene2}
\end{equation}

Also, for $r=1$, we obtain the recurrence relation of the $s$-associated
Stirling numbers \cite[eq 4.1]{MR600368}.
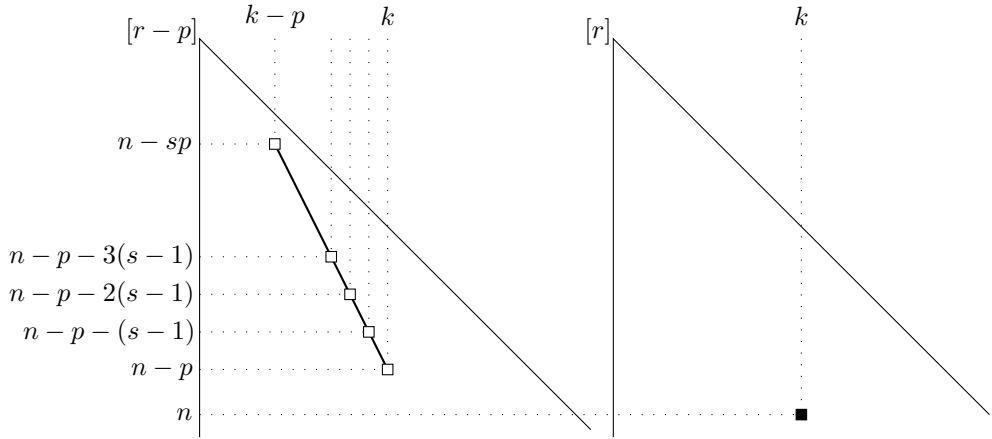
\begin{figure}[!h]
\centering
\begin{tikzpicture}
	\draw (0,-0.3)--(0,5)--(5.2,-0.2);
	\draw[thick] (2.5,0.6)--(1,3.6);
	\draw[thin,loosely dotted] (0,0.6)--(2.5,0.6);
	\draw[thin,loosely dotted] (2.5,5)--(2.5,0.6);
	\draw[fill=white] (2.43,0.53) rectangle(2.57,0.67);
	\node  at (2.5,5.3) {$k$};
	\node  at (-0.5,0.6) {$n-p$};
	
	\draw[thin,loosely dotted] (0,1.1)--(2.25,1.1) ; 
	\draw[thin,loosely dotted] (2.25,5)--(2.25,1.1);
	\draw[fill=white] (2.18,1.03) rectangle(2.32,1.17);
	\node  at (-1.2,1.1) {$n-p-(s-1)$};
	\draw[thin,loosely dotted] (0,1.6)--(2.0,1.6) ; 
	\draw[thin,loosely dotted] (2.0,5)--(2.0,1.6);
	\draw[fill=white] (1.93,1.53) rectangle(2.07,1.67);
	\node  at (-1.3,2.1) {$n-p-3(s-1)$};
	\draw[thin,loosely dotted] (0,2.1)--(1.75,2.1) ; 
	\draw[thin,loosely dotted] (1.75,5)--(1.75,2.1);
	\draw[fill=white] (1.68,2.03) rectangle(1.82,2.17);
	\node  at (-1.3,1.6) {$n-p-2(s-1)$};
	\draw[thin,loosely dotted] (0,3.6)--(1,3.6) ; 
	\draw[thin,loosely dotted] (1,5)--(1,3.6);
	\draw[fill=white] (0.93,3.53) rectangle(1.07,3.67);
	\node  at (1,5.3) {$k-p$};
	\node  at (-0.6,3.6) {$n-sp$};
	
	\draw (5.5,-0.3)--(5.5,5)--(10.5,0);
	\draw[thin,loosely dotted] (0,0)--(8,0.0);
	\draw[thin,loosely dotted] (8,5)--(8,0.0);
	\draw[fill=black] (7.93,-0.07) rectangle(8.07,0.07);
	\node  at (8,5.3) {$k$};
	\node  at (-0.2,0.0) {$n$};
	
	\node  at (5.3,5.1) {$[r]$};
	\node  at (-0.5,5.1) {$[r-p]$};

\end{tikzpicture}
\caption{Value of $s$-associated $r$-Stirling element (in black) as a inner product of a periodic sequence of elements of the $s$-associated $r-p$-Stirling table (in white) with a sequence deriving from binomial coefficient.}
\label{Tull}
\end{figure}
\begin{theorem}
Let $p,r,k$ and $n$ be nonnegative integers such that $p\leqslant r\leqslant
k$ and $n\geq sk$, we have%
\begin{equation}
\lah
{n}{k}_{r}^{(s)}=\sum\limits_{i=0}^{p}\sum_{j=i}^{n-sk}\frac{\left(
n-r\right) !}{\left( n-r-j+\left( s-1\right) p\right) !}\dbinom{p}{i}\dbinom{%
p+j-1}{j-i}\lah {n-sp-j}{k-p}_{r-p}^{(s)}s^{p-i}.
\label{TooLong}
\end{equation}
\end{theorem}

\begin{proof}
Let us consider the $p$ first elements, they are in $p$ distinct lists with $%
i_{j}$ $\left( i_{j}\geqslant s-1;\text{ }j=1..p\right) $ other elements,
such that $i_{1}+i_{2}+\cdots i_{p}=j$ $\left( \left( s-1\right) p\leqslant
j\leqslant n-p-s\left( k-p\right) \right) .$ Then there are $\binom{n-r}{%
i_{1},i_{2},\ldots ,i_{p},n-r-j}=\binom{n-r}{i_{1}}\binom{n-r-i_{1}}{i_{2}}%
\cdots \binom{n-r-i_{1}-i_{2}-\ldots -i_{p-1}}{i_{p}}$ ways to choose the $%
i_{1},i_{2},\ldots ,i_{p}$ elements and $(i_{1}+1)!(i_{2}+1)!\cdots
(i_{p}+1)!$ ways to constitute the $p$ lists. Now, it remains to distribute
the $n-p-j$ remaining elements into $k-p$ lists such that each list have at
least $s$ elements \ and the $r-p$ elements are in distinct lists, which
gives $\lah {n-p-j}{k-p}_{r-p}^{(s)}$ possibilities. we
sum over all value of $j$ we get 
\begin{equation*}
\lah {n}{k}_{r}^{(s)}=\sum_{j=\left( s-1\right)
p}^{n-p-s\left( k-p\right) }\sum\limits_{\substack{ i_{1}+i_{2}+\cdots
+i_{p}=j  \\ i_{l\geqslant s-1}}}(i_{1}+1)!(i_{2}+1)!\cdots (i_{p}+1)!\binom{%
n-r}{i_{1},i_{2},\ldots ,i_{p},n-r-j}\lah
{n-p-j}{k-p}_{r-p}^{(s)},
\end{equation*}%
the inner summations can be evaluated using (\ref{Getit}). This gives the
result.
\end{proof}

For $p=r$, we get an expression of the $s$-associated $r$-Lah numbers in
terms of the $s$-associated Lah numbers%
\begin{equation}
\lah
{n}{k}_{r}^{(s)}=\sum_{i=0}^{r}\sum_{j=i}^{n-sk}\dbinom{r}{i}\dbinom{%
r+j-1}{j-i}\frac{\left( n-r\right) !}{\left( n-j+\left( s-2\right) r\right) !%
}s^{r-i}\lah {n-sr-j}{k-r}^{(s)}.
\end{equation}

Also, For $r=1,$ we get relation (\ref{RecVer}), and for $s=1$ we get the
identity \cite[eq 7]{BelBel13}.

\section{Cross recurrence relations}

From equations (\ref{S1rsfiS1p}) and (\ref{TooLong}), for $p=1$, we get some
vertical cross recurrence relations.

\begin{corollary}
We have%
\begin{equation}
\stirlingf {n}{k}_{r}^{(s)}=\sum_{i=s-1}^{n-s(k-1)-1}\frac{\left(
n-r\right) !}{\left( n-r-i\right) !}\stirlingf {n-i-1}{k-1}_{r-1}^{(s)},
\end{equation}%
\begin{equation}
\lah {n}{k}_{r}^{(s)}=\sum_{i=s-1}^{n-s(k-1)-1}(i+1)\frac{%
\left( n-r\right) !}{\left( n-r-i\right) !}\lah
{n-i-1}{k-1}_{r-1}^{(s)}.
\end{equation}
\end{corollary}

For $r=1,$ we get relation (\ref{RecVer}) and for $s=1$ we get the identity
given by the authors in \cite{BelBou14}.

\begin{theorem}
Let $r,k,n$ be nonnegative integers such that $n\geq sk$, we have%
\begin{equation}
\stirlings {n}{k}_{r}^{(s)}=\binom{n-r}{s-1}\stirlings
{n-s}{k-1}_{r-1}^{(s)}+\left( k-r+1\right) \stirlings {n-1}{k}_{r-1}^{(s)}.
\end{equation}
\end{theorem}

\begin{proof}
Let us consider the $r^{th}$ elements. If it belongs to a group containing
exactly $s$ elements, we have $\binom{n-r}{s-1}$ ways to choose the
remaining $\left( s-1\right) $ elements and $\stirlings
{n-s}{k-1}_{r-1}^{(s)}$ ways to partition the remaining $\left( n-s\right) $
elements on $\left( k-1\right) $ parts such that the $\left( r-1\right) $
first elements are in distinct parts, and each parts, have at least $s$
elements. Thus we get $\binom{n-r}{s-1}\stirlings {n-s}{k-1}_{r-1}^{(s)}$
possibilities. Else, we have $\stirlings {n-1}{k}_{r-1}^{(s)}$
possibilities to partition the remaining $\left( n-1\right) $ elements into $%
k$ parts such that the $\left( r-1\right) $ first elements are in distinct
parts, and each parts, have at least $s$ elements, then add the $r^{th}$
elements to on of the $(k-(r-1))$ parts and we have $\left( k-r+1\right) $
possibilities. It gives $\left( k-r+1\right) \stirlings
{n-1}{k}_{r-1}^{(s)}$.
\end{proof}

For $s=1$ we get the cross recurrence \cite[eq 3]{BelBou12} and for $r=1$ we
get the recurrence relation of the $s$-associated Stirling numbers of the
second kind \cite[eq 4.1]{MR600368}.

\section{Convolution identities (revisited)}

The $s$-associated $r$-Stirling numbers of the three kinds can be expressed
as a convolution using the binomial coefficients.

\begin{theorem}
Let $r,k$ and $n$ be nonnegative integers such that $n\geq sk$ with $%
k_{1}+\cdots +k_{p}=k$ and $r_{1}+\cdots +r_{p}=r$, we have%
\begin{equation}
\binom{k}{k_{1},\ldots ,k_{p}}\stirlingf {n+r}{k+r}_{r}^{(s)}=\sum\limits 
_{\substack{ l_{1}+\cdots +l_{p}=n  \\ l_{i}\geqslant sk_{i}+(s-1)r_{i}}}%
\binom{n}{l_{1},\ldots ,l_{p}}\stirlingf {l_{1}+r_{1}}{k_{1}+r_{1}}%
_{r_{1}}^{(s)}\cdots \stirlingf {l_{p}+r_{p}}{k_{p}+r_{p}}_{r_{p}}^{(s)}.
\label{didi}
\end{equation}
\end{theorem}

\begin{proof}
We consider permutations of $Z_{n+r}$ with $k+r$ cycles such that the $r$
first elements are in distinct cycles and each cycle has at least $s$
elements and we have $\stirlingf {n+r}{k+r}_{r}^{(s)}$ possibilities. We
color the cycles with $p$ colors such that each $r_{i}$ cycles containing
the $r_{i}$ elements with $k_{i}$ other cycles have the same color, thus we
choose the $k_{i}$ cycles and we have $\binom{k}{k_{1},\ldots ,k_{p}}$
possibilities this is to choose the $l_{i}$ elements that have the same
color of the $r_{i}$ first and we have $\binom{n}{l_{1},\ldots ,l_{p}}$
possibilities, then consider all the permutations of the $l_{i}+r_{i}$
elements with $k_{i}+r_{i}$ cycles such that the $r_{i}$ elements are in
distinct cycles and each cycle has at least $s$ element and we have $\stirlingf {l_{i}+r_{i}}{k_{i}+r_{i}}_{ri}^{(s)}$ ways to do it. Summing over all
possible values of $l_{i}$ gives the result.
\end{proof}

\begin{theorem}
Let $r,k$ and $n$ be nonnegative integers such that $n\geq sk$ with $%
k_{1}+\cdots +k_{p}=k$ and $r_{1}+\cdots +r_{p}=r$, The $s$-associated $r$%
-Stirling numbers of the second kind satisfy%
\begin{equation}
\binom{k}{k_{1},\ldots ,k_{p}}\stirlings {n+r}{k+r}_{r}^{(s)}=\sum\limits 
_{\substack{ l_{1}+\cdots +l_{p}=n  \\ l_{i}\geqslant sk_{i}+(s-1)r_{i}}}%
\binom{n}{l_{1},\ldots ,l_{p}}\stirlings
{l_{1}+r_{1}}{k_{1}+r_{1}}_{r_{1}}^{(s)}\cdots \stirlings
{l_{p}+r_{p}}{k_{p}+r_{p}}_{r_{p}}^{(s)}.  \label{dd}
\end{equation}
\end{theorem}

\begin{proof}
We use an adapted analogous bijective proof as for the identity (\ref{didi}).
\end{proof}

Relations (\ref{didi}) and (\ref{dd}) extend those given by the others \cite[%
Eq 8, Eq 12]{BelBou12} to the $s$-associated situation.

\begin{theorem}
Let $r,k$ and $n$ be nonnegative integers such that $n\geq sk$ with $%
k_{1}+\cdots +k_{p}=k$ and $r_{1}+\cdots +r_{p}=r$, The $s$-associated $r$%
-Lah numbers satisfy%
\begin{equation}
\binom{k}{k_{1},\ldots ,k_{p}}\lah
{n+r}{k+r}_{r}^{(s)}=\sum\limits_{\substack{ l_{1}+\cdots +l_{p}=n  \\ %
l_{i}\geqslant sk_{i}+(s-1)r_{i}}}\binom{n}{l_{1},\ldots ,l_{p}}%
\lah {l_{1}+r_{1}}{k_{1}+r_{1}}_{r_{1}}^{(s)}\cdots
\lah {l_{p}+r_{p}}{k_{p}+r_{p}}_{r_{p}}^{(s)}.
\end{equation}
\end{theorem}

\begin{proof}
We use an adapted analogous bijective proof as for the identity (\ref{didi}).
\end{proof}

For $s=1$, we get 
\begin{equation}
\binom{k}{k_{1},\ldots ,k_{p}}\lah
{n+r}{k+r}_{r}=\sum\limits_{\substack{ l_{1}+\cdots +l_{p}=n  \\ %
l_{i}\geqslant sk_{i}+(s-1)r_{i}}}\binom{n}{l_{1},\ldots ,l_{p}}%
\lah {l_{1}+r_{1}}{k_{1}+r_{1}}_{r_{1}}\cdots
\lah {l_{p}+r_{p}}{k_{p}+r_{p}}_{r_{p}}.
\end{equation}

\section{Generating functions}

The $s$-associated $r$-Stirling numbers of the first kind have the following
exponential generating function.

\begin{theorem}
We have%
\begin{equation}
\sum\limits_{n\geq sk+(s-1)r}\stirlingf {n+r}{k+r}_{r}^{(s)}\frac{x^{n}}{n!}=%
\frac{\left( -1\right) ^{k}}{k!}\left( \ln \left( 1-x\right)
+\sum_{i=1}^{s-1}\frac{x^{i}}{i}\right) ^{k}\left( \frac{x^{s-1}}{1-x}%
\right) ^{r}.  \label{generatingS1}
\end{equation}
\end{theorem}

\begin{proof}
Using the identity (\ref{S1rsS1s}), we get%
\begin{equation*}
\sum\limits_{n\geq sk+(s-1)r}\stirlingf {n+r}{k+r}_{r}^{(s)}\frac{x^{n}}{n!}%
=\sum\limits_{i}\dbinom{i-r\left( s-2\right) -1}{r-1}x^{i}\sum\limits_{n\geq
sk+(s-1)r}\stirlingf {n-i}{k}^{(s)}\frac{x^{n-i}}{\left( n-i\right) !},
\end{equation*}%
from (\ref{genes1}), we obtain%
\begin{equation*}
\sum\limits_{n\geq sk+(s-1)r}\stirlingf {n+r}{k+r}_{r}^{(s)}\frac{x^{n}}{n!}=%
\frac{1}{k!}\left( \ln \left( \frac{1}{1-x}\right) -\sum\limits_{i=1}^{s-1}%
\frac{x^{i}}{i}\right) ^{k}\sum\limits_{i}\dbinom{i-r\left( s-2\right) -1}{%
r-1}x^{i},
\end{equation*}%
using relation (\ref{y}) we get the result.
\end{proof}

The above theorem implies the double generating function.

\begin{theorem}
The $s$-associated $r$-Stirling numbers of the first kind satisfy%
\begin{equation}
\sum\limits_{n,k}\stirlingf {n+r}{k+r}_{r}^{(s)}y^{k}\frac{x^{n}}{n!}=\exp
\left( y\ln \left( \frac{1}{{1-x}}\right) -y\sum_{i=1}^{s-1}\frac{x^{i}}{i}%
\right) \left( \frac{x^{s-1}}{1-x}\right) ^{r}.
\end{equation}
\end{theorem}

\begin{proof}
Interchanging the order of summation and using equation (\ref{generatingS1})
we get the result.
\end{proof}

The $s$-associated $r$-Stirling numbers of the second kind have the
following exponential generating function

\begin{theorem}
We have%
\begin{equation}
\sum\limits_{n\geq sk+(s-1)r}\stirlings {n+r}{k+r}_{r}^{(s)}\frac{x^{n}}{n!%
}=\frac{1}{k!}\left( \exp \left( x\right) -\sum_{i=0}^{s-1}\frac{x^{i}}{i!}%
\right) ^{k}\left( \exp \left( x\right) -\sum_{i=0}^{s-2}\frac{x^{i}}{i!}%
\right) ^{r}.  \label{generatingS2}
\end{equation}
\end{theorem}

\begin{proof}
Using the identity (\ref{gene2}), we get%
\begin{equation*}
\sum\limits_{n\geq sk+(s-1)r}\stirlings {n+r}{k+r}_{r}^{(s)}\frac{x^{n}}{n!%
}=\frac{(k+r-i)!}{k!}\sum_{i=0}^{r}\binom{r}{i}\left( \frac{x^{s-1}}{\left(
\left( s-1\right) !\right) }\right) ^{i}\sum\limits_{n\geq
sk+(s-1)r}\stirlings {n-i(s-1)}{k+r-i}^{(s)}\frac{x^{n-i(s-1)}}{\left(
n-i(s-1)\right) !},
\end{equation*}%
the second summation can be evaluated using (\ref{genes2}) and gives%
\begin{equation*}
\sum\limits_{n\geq sk+(s-1)r}\stirlings {n+r}{k+r}_{r}^{(s)}\frac{x^{n}}{n!%
}=\frac{1}{k!}\left( \exp \left( x\right) -\sum_{i=0}^{s-1}\frac{x^{i}}{i!}%
\right) ^{k}\sum_{i=0}^{r}\binom{r}{i}\left( \frac{x^{s-1}}{\left( \left(
s-1\right) !\right) }\right) ^{i}\left( \exp \left( x\right)
-\sum_{i=0}^{s-1}\frac{x^{i}}{i!}\right) ^{r-i},
\end{equation*}%
using the binomial theorem we get the result.
\end{proof}

The double generating function for $s$-associated $r$-Stirling numbers of
the second kind is

\begin{theorem}
\begin{equation}
\sum\limits_{n,k}\stirlings {n+r}{k+r}_{r}^{(s)}y^{k}\frac{x^{n}}{n!}=\exp
\left( y\exp \left( x\right) -y\sum_{i=0}^{s-1}\frac{x^{i}}{i!}\right)
\left( \exp \left( x\right) -\sum_{i=0}^{s-1}\frac{x^{i}}{i!}\right) ^{r}.
\end{equation}
\end{theorem}

The $s$-associated $r$-Lah numbers have the following exponential generating
function

\begin{theorem}
We have%
\begin{equation}
\sum\limits_{n\geq sk+(s-1)r}\lah {n+r}{k+r}_{r}^{(s)}%
\frac{x^{n}}{n!}=\frac{1}{k!}\left( \frac{x^{s}}{\left( 1-x\right) }\right)
^{k}\left( \frac{x^{s-1}}{\left( 1-x\right) ^{2}}(s-(s-1)x)\right) ^{r}.
\label{gninin}
\end{equation}
\end{theorem}

\begin{proof}
Using the explicit formula (\ref{exp2}) in the left hand side we get%
\begin{equation*}
\sum\limits_{n\geq sk+(s-1)r}\lah {n+r}{k+r}_{r}^{(s)}%
\frac{x^{n}}{n!}=\frac{1}{k!}\sum\limits_{j=0}^{r}\dbinom{r}{j}%
(s-1)^{r-j}\sum\limits_{n\geq sk+(s-1)r}\dbinom{n+r+j-(s-1)\left( k+r\right)
-1}{k+r+j-1}x^{n},
\end{equation*}%
the second summation in the right side, due to relation (\ref{y}), gives%
\begin{eqnarray*}
\sum\limits_{n\geq sk+(s-1)r}\lah {n+r}{k+r}_{r}^{(s)}%
\frac{x^{n}}{n!} &=&\frac{1}{k!}\sum\limits_{j=0}^{r}\dbinom{r}{j}(s-1)^{r-j}%
\frac{x^{(s-1)(k+r)+k}}{\left( 1-x\right) ^{k+r+j}} \\
&=&\frac{1}{k!}\frac{x^{(s-1)r+sk}}{(1-x)^{k+2r}}\sum\limits_{j=0}^{r}%
\dbinom{r}{j}\left( (s-1)(1-x)\right) ^{r-j}
\end{eqnarray*}%
using the binomial theorem we get the result.
\end{proof}

The double generating function of the $s$-associated $r$-Lah numbers is
given by

\begin{theorem}
\begin{equation}
\sum\limits_{n\geq sk+(s-1)r}\sum\limits_{k\geq 0}\lah
{n+r}{k+r}_{r}^{(s)}\frac{x^{n}}{n!}y^{k}=\left[ \exp \left\{ y\frac{x^{s}}{%
1-x}\right\} \right] \left[ \frac{x^{s-1}}{(1-x)^{2}}(s-(s-1)x)\right] ^{r}.
\end{equation}
\end{theorem}

\begin{proof}
Interchanging the order of summation and using equation (\ref{gninin}) we
get the result.
\end{proof}

\section{Conclusion and perspectives}

Roughly speaking, there are many recurrence and congruence relations known
about the $r$-Stirling numbers and the associated Stirling numbers which can
be generalized to $\stirlingf {n}{k}_{r}^{(s)}$, $\stirlings
{n}{k}_{r}^{(s)}$ and $\lah {n}{k}_{r}^{(s)}$. We have
treated a few of these. In this section, we propose some problems :

\begin{itemize}
\item Howard \cite{MR742846} gave, as perspectives, an extension of the
weighted associated Stirling numbers to the Weighted $s$-associated Stirling
numbers without specifying the expressions. In this perspective, as
continuity to our work, we propose the Weighted $s$-associated Stirling
numbers of the first and the second kind, denoted $\stirlingf {n}{k}%
_{\lambda }^{(s)}$ and $\stirlings {n}{k}_{\lambda }^{(s)}$ respectively,
by the following%
\begin{eqnarray}
\sum\limits_{n\geq sk}\stirlingf {n}{k}_{\lambda }^{(s)}\frac{x^{n}}{n!} &=&%
\frac{\left( -1\right) ^{k}}{k!}\left( \frac{1}{\left( 1-x\right) ^{\lambda }%
}-\sum_{i=1}^{s-1}\frac{\left( -\lambda \right) _{i}x^{i}}{i}\right) \left(
\ln \left( 1-x\right) +\sum_{i=1}^{s-1}\frac{x^{i}}{i}\right) ^{k}, \\
\sum\limits_{n\geq s}\stirlings {n}{k}_{\lambda }^{(s)}\frac{x^{n}}{n!} &=&%
\frac{1}{k!}\left( \exp \left( \lambda x\right) -\sum_{i=0}^{s-1}\frac{%
\left( \lambda x\right) ^{i}}{i!}\right) \left( \exp \left( x\right)
-\sum_{i=0}^{s-1}\frac{x^{i}}{i!}\right) ^{k}.
\end{eqnarray}%
\newline
Note that for $s=2$, we get weighted associated Stirling numbers. It seems
possible to derive analog relations of the weighted associated Stirling
numbers, and establish other identities.

\item By the same reasoning, it is interesting to extend these
generalization to the $r$-Stirling numbers. We define the weighted $s$%
-associated $r$-Stirling numbers of the first and the second kind
respectively%
\begin{eqnarray}
\sum_{n}\stirlingf {n+r}{k+r}_{r,\lambda }^{(s)}\frac{%
x^{n}}{n!} &=&\frac{\left( -1\right) ^{k}}{k!}\left( \frac{1}{\left(
1-x\right) ^{\lambda }}-\sum_{i=1}^{s-1}\frac{\left( -\lambda \right)
_{i}x^{i}}{i}\right) \left( \ln \left( 1-x\right) +\sum_{i=1}^{s-1}\frac{%
x^{i}}{i}\right) ^{k}\left( \frac{x^{s-1}}{1-x}\right) ^{r}, \\
\sum_{n}\stirlings {n+r}{k+r}_{r,\lambda }^{(s)}\frac{%
x^{n}}{n!} &=&\frac{1}{k!}\left( \exp \left( \lambda x\right)
-\sum_{i=0}^{s-1}\frac{\left( \lambda x\right) ^{i}}{i!}\right) \left( \exp
\left( x\right) -\sum_{i=0}^{s-1}\frac{x^{i}}{i!}\right) ^{k}\left( \exp
\left( x\right) -\sum_{i=0}^{s-2}\frac{x^{i}}{i!}\right) ^{r}.
\end{eqnarray}%
\newline
It will be nice to investigate the combinatorial meaning and drive all the
combinatorial identities. Also, for $s=1$, we get the definition of the
weighted $r$-Stirling numbers as follows 
\begin{eqnarray}
\sum\limits_{n\geq sk+(s-1)r}\stirlingf {n+r}{k+r}_{r,\lambda }\frac{x^{n}}{%
n!} &=&\frac{\left( -1\right) ^{k}}{k!}\frac{1}{\left( 1-x\right) ^{\lambda
+r}}\left( \ln \left( 1-x\right) \right) ^{k}, \\
\sum\limits_{n\geq sk+(s-1)r}\stirlings {n+r}{k+r}_{r,\lambda }\frac{x^{n}%
}{n!} &=&\frac{1}{k!}\left( \exp \left( \lambda x\right) -1\right) \left(
\exp \left( x\right) -1\right) ^{k}\exp \left( rx\right) .
\end{eqnarray}

\item It will be nice to investigate the different generalization (weighted,
degenerated) of the Lah numbers and $r$-Lah numbers.

\item The authors and Belkhir \cite{BelBou141} define the $\lah {n}{k}^{\alpha ,\beta }$ as the weight of a partition of $n$
elements into $k$ lists such that the element inserted as head list has
weight $\beta $ except the first inserted one which has weight $1$ and the
element inserted after an other one has weight $\alpha $. This
interpretation allow an extension to the $s$-associated aspect by adding the
known restriction (at least $s$ elements by list).

\item An other perspective of this work is to consider the Whitney numbers (see \cite%
{MR1415279,MR1453407,MR1659426,MR2900003}) and $r$-Whitney numbers (see \cite%
{MR2926106}) and to introduce the $s$-associated situation by two approaches: the first one via the generating function and the second one using the combinatorial interpretation (see \cite{BelBou142}).
\end{itemize}

\section{Tables of the $s$-associated $r$-Stirling numbers of the three kinds%
}

\begin{center}
\begin{equation*}
\begin{tabular}{c|cccccc}
\hline
$n\backslash k$ & $3$ & $4$ & $5$ & $6$ & $7$ & $8$ \\ \hline
$6$ & $6$ &  &  &  &  &  \\ 
$7$ & $72$ &  &  &  &  &  \\ 
$8$ & $720$ & $60$ &  &  &  &  \\ 
$9$ & $7200$ & $1320$ &  &  &  &  \\ 
$10$ & $75\,600$ & $21\,420$ & $630$ &  &  &  \\ 
$11$ & $846\,720$ & $320\,544$ & $21\,840$ &  &  &  \\ 
$12$ & $10\,160\,640$ & $4753\,728$ & $519\,120$ & $7560$ &  &  \\ 
$13$ & $130\,636\,800$ & $72\,005\,760$ & $10\,795\,680$ & $378\,000$ &  & 
\\ 
$14$ & $1796\,256\,000$ & $1129\,788\,000$ & $213\,804\,360$ & $12\,335\,400$
& $103\,950$ &  \\ 
$15$ & $26\,\allowbreak 345\,088\,000$ & $18\,\allowbreak 486\,230\,400$ & $%
4191\,881\,760$ & $339\,255\,840$ & $7068\,600$ &  \\ 
$16$ & $410\,\allowbreak 983\,372\,800$ & $316\,\allowbreak 406\,787\,840$ & 
$83\,\allowbreak 018\,048\,256$ & $8627\,739\,120$ & $302\,702\,400$ & $%
1621\,620$ \\ 
$17$ & $6799\,\allowbreak 906\,713\,600$ & $5670\,\allowbreak 985\,582\,080$
& $1679\,\allowbreak 434\,428\,672$ & $212\,\allowbreak 106\,454\,560$ & $%
10\,\allowbreak 621\,490\,880$ & $143\,783\,640$ \\ \hline
\end{tabular}%
\end{equation*}%
Some values of the $2$-associated $3$-Stirling numbers of the first kind%
\begin{equation*}
\begin{tabular}{c|ccccc}
\hline
$n\backslash k$ & $2$ & $3$ & $4$ & $5$ & $6$ \\ \hline
$6$ & $24$ &  &  &  &  \\ 
$7$ & $240$ &  &  &  &  \\ 
$8$ & $2160$ &  &  &  &  \\ 
$9$ & $20\,160$ & $\allowbreak 1680$ &  &  &  \\ 
$10$ & $201\,600$ & $\allowbreak 36\,960$ &  &  &  \\ 
$11$ & $2177\,280$ & $616\,896$ &  &  &  \\ 
$12$ & $25\,401\,600$ & $9616\,320$ & $201\,600$ &  &  \\ 
$13$ & $319\,334\,400$ & $145\,774\,080$ & $7761\,600$ &  &  \\ 
$14$ & $4311\,014\,400$ & $2329\,015\,680$ & $206\,569\,440$ &  &  \\ 
$15$ & $62\,\allowbreak 270\,208\,000$ & $39\,\allowbreak 165\,984\,000$ & $%
4817\,292\,480$ & $38\,438\,400$ &  \\ 
$16$ & $958\,\allowbreak 961\,203\,200$ & $672\,\allowbreak 898\,786\,560$ & 
$106\,\allowbreak 815\,893\,184$ & $2287\,084\,800$ &  \\ 
$17$ & $15\,692\,\allowbreak 092\,416\,000$ & $12\,080\,\allowbreak
986\,444\,800$ & $2337\,\allowbreak 623\,608\,320$ & $88\,\allowbreak
691\,803\,200$ &  \\ 
$18$ & $271\,996\,\allowbreak 268\,544\,000$ & $226\,839\,\allowbreak
423\,283\,200$ & $51\,485\,\allowbreak 284\,730\,880$ & $2886\,\allowbreak
166\,483\,200$ & $10\,\allowbreak 762\,752\,000$ \\ 
$19$ & $4979\,623\,\allowbreak 993\,344\,000$ & $4453\,872\,\allowbreak
650\,035\,200$ & $1153\,763\,\allowbreak 447\,316\,480$ & $%
86\,362\,\allowbreak 805\,168\,640$ & $914\,\allowbreak 833\,920\,000$ \\ 
\hline
\end{tabular}%
\end{equation*}%
Some values of the $3$-associated $2$-Stirling numbers of the first kind%
\begin{equation*}
\begin{tabular}{c|cccc}
\hline
$n\backslash k$ & $3$ & $4$ & $5$ & $6$ \\ \hline
$9$ & $720$ &  &  &  \\ 
$10$ & $15\,120$ &  &  &  \\ 
$11$ & $241\,920$ &  &  &  \\ 
$12$ & $3628\,800$ & $120\,960$ &  &  \\ 
$13$ & $54\,432\,000$ & $4536\,000$ &  &  \\ 
$14$ & $838\,252\,800$ & $117\,754\,560$ &  &  \\ 
$15$ & $13\,412\,044\,800$ & $2682\,408\,960$ & $26\,611\,200$ &  \\ 
$16$ & $224\,172\,748\,800$ & $57\,916\,892\,160$ & $1556\,755\,200$ &  \\ 
$17$ & $3923\,023\,104\,000$ & $1239\,100\,934\,400$ & $59\,\allowbreak
390\,210\,880$ &  \\ 
$18$ & $71\,922\,090\,240\,000$ & $26\,544\,536\,282\,880$ & $%
1902\,\allowbreak 484\,584\,000$ & $8072\,064\,000$ \\ 
$19$ & $1380\,904\,\allowbreak 132\,608\,000$ & $592\,364\,034\,662\,400$ & $%
56\,075\,\allowbreak 567\,708\,160$ & $678\,\allowbreak 053\,376\,000$ \\ 
$20$ & $27\,743\,619\,\allowbreak 391\,488\,000$ & $13\,356\,216\,902\,246%
\,400$ & $1589\,118\,\allowbreak 272\,501\,760$ & $35\,651\,\allowbreak
077\,862\,400$ \\ \hline
\end{tabular}%
\end{equation*}

Some values of the $3$-associated $3$-Stirling numbers of the first kind

\begin{equation*}
\begin{tabular}{c|cccccc}
\hline
$n\backslash k$ & $3$ & $4$ & $5$ & $6$ & $7$ & $8$ \\ \hline
$6$ & $6$ &  &  &  &  &  \\ 
$7$ & $36$ &  &  &  &  &  \\ 
$8$ & $150$ & $60$ &  &  &  &  \\ 
$9$ & $540$ & $660$ &  &  &  &  \\ 
$10$ & $1806$ & $4620$ & $630$ &  &  &  \\ 
$11$ & $5796$ & $26\,376$ & $10\,920$ &  &  &  \\ 
$12$ & $18\,150$ & $134\,316$ & $114\,660$ & $7560$ &  &  \\ 
$13$ & $55\,980$ & $637\,020$ & $947\,520$ & $189\,000$ &  &  \\ 
$14$ & $171\,006$ & $2882\,220$ & $6798\,330$ & $2772\,000$ & $103\,950$ & 
\\ 
$15$ & $519\,156$ & $12\,623\,952$ & $44\,482\,680$ & $31\,221\,960$ & $%
3534\,300$ &  \\ 
$16$ & $1569\,750$ & $54\,031\,692$ & $273\,060\,216$ & $299\,459\,160$ & $%
68\,918\,850$ & $1621\,620$ \\ 
$17$ & $4733\,820$ & $227\,425\,380$ & $1600\,815\,216$ & $2578\,495\,920$ & 
$1013\,632\,620$ & $71\,891\,820$ \\ 
$18$ & $14\,250\,606$ & $945\,535\,500$ & $9069\,810\,750$ & $%
20\,\allowbreak 561\,420\,880$ & $12\,\allowbreak 509\,597\,100$ & $%
1797\,295\,500$ \\ 
$19$ & $42\,850\,116$ & $3895\,163\,928$ & $50\,\allowbreak 074\,806\,600$ & 
$154\,\allowbreak 904\,109\,360$ & $136\,\allowbreak 912\,175\,400$ & $%
33\,\allowbreak 423\,390\,000$ \\ \hline
\end{tabular}%
\end{equation*}%
$\allowbreak $Some values of the $2$-associated $3$-Stirling numbers of the
second kind

\begin{equation*}
\begin{tabular}{c|cccccc}
\hline
$n\backslash k$ & $2$ & $3$ & $4$ & $5$ & $6$ & $7$ \\ \hline
$6$ & $6$ &  &  &  &  &  \\ 
$7$ & $20$ &  &  &  &  &  \\ 
$8$ & $50$ &  &  &  &  &  \\ 
$9$ & $112$ & $210$ &  &  &  &  \\ 
$10$ & $238$ & $1540$ &  &  &  &  \\ 
$11$ & $492$ & $7476$ &  &  &  &  \\ 
$12$ & $1002$ & $30\,240$ & $12\,600$ &  &  &  \\ 
$13$ & $2024$ & $110\,550$ & $161\,700$ &  &  &  \\ 
$14$ & $4070$ & $379\,764$ & $1286\,670$ &  &  &  \\ 
$15$ & $8164$ & $1252\,680$ & $8168\,160$ & $1201\,200$ &  &  \\ 
$16$ & $16\,354$ & $4020\,016$ & $45\,411\,366$ & $23\,823\,800$ &  &  \\ 
$17$ & $32\,736$ & $12\,656\,826$ & $231\,591\,360$ & $281\,331\,050$ &  & 
\\ 
$18$ & $65\,502$ & $39\,315\,588$ & $1112\,731\,620$ & $2574\,371\,800$ & $%
168\,168\,000$ &  \\ 
$19$ & $131\,036$ & $120\,953\,436$ & $5122\,253\,136$ & $\allowbreak
20\,\allowbreak 176\,035\,880$ & $4764\,760\,000$ &  \\ 
$20$ & $262\,106$ & $369\,535\,392$ & $22\,\allowbreak 845\,529\,356$ & $%
142\,\allowbreak 501\,719\,360$ & $78\,\allowbreak 189\,711\,600$ &  \\ 
$21$ & $524\,248$ & $1123\,340\,382$ & $99\,\allowbreak 494\,683\,548$ & $%
934\,\allowbreak 588\,410\,756$ & $973\,\allowbreak 654\,882\,200$ & $%
32\,\allowbreak 590\,958\,400$ \\ \hline
\end{tabular}%
\end{equation*}%
$\allowbreak $Some values of the $3$-associated $2$-Stirling numbers of the
second kind

\begin{equation*}
\begin{tabular}{c|ccccc}
\hline
$n\backslash k$ & $3$ & $4$ & $5$ & $6$ & $7$ \\ \cline{2-6}
$9$ & $90$ &  &  &  &  \\ 
$10$ & $630$ &  &  &  &  \\ 
$11$ & $2940$ &  &  &  &  \\ 
$12$ & $11\,508$ & $7560$ &  &  &  \\ 
$13$ & $40\,950$ & $94\,500$ &  &  &  \\ 
$14$ & $137\,610$ & $734\,580$ &  &  &  \\ 
$15$ & $445\,896$ & $4569\,180$ & $831\,600$ &  &  \\ 
$16$ & $1410\,552$ & $24\,959\,220$ & $16\,216\,200$ &  &  \\ 
$17$ & $4390\,386$ & $125\,381\,256$ & $188\,558\,370$ &  &  \\ 
$18$ & $13\,514\,046$ & $594\,714\,120$ & $\allowbreak 1701\,649\,950$ & $%
126\,126\,000$ &  \\ 
$19$ & $41\,278\,068$ & $2707\,865\,160$ & $13\,\allowbreak 172\,479\,320$ & 
$3531\,528\,000$ &  \\ 
$20$ & $125\,405\,532$ & $11\,\allowbreak 965\,834\,608$ & $92\,\allowbreak
024\,532\,600$ & $57\,\allowbreak 320\,062\,800$ &  \\ 
$21$ & $379\,557\,198$ & $51\,\allowbreak 706\,343\,676$ & $597\,\allowbreak
753\,095\,940$ & $706\,\allowbreak 637\,731\,800$ & $25\,\allowbreak
729\,704\,000$ \\ 
$22$ & $1145\,747\,538$ & $219\,\allowbreak 672\,404\,652$ & $%
3679\,\allowbreak 670\,518\,524$ & $7344\,\allowbreak 721\,664\,280$ & $%
977\,\allowbreak 728\,752\,000$ \\ 
$23$ & $3452\,182\,656$ & $921\,\allowbreak 197\,481\,924$ & $%
21\,746\,\allowbreak 705\,483\,880$ & $67\,927\,\allowbreak 123\,063\,800$ & 
$21\,102\,\allowbreak 645\,564\,000$ \\ 
$24$ & $10\,\allowbreak 388\,002\,848$ & $3824\,\allowbreak 306\,218\,236$ & 
$124\,527\,\allowbreak 413\,730\,720$ & $577\,211\,\allowbreak 131\,256\,760$
& $340\,398\,\allowbreak 980\,922\,000$ \\ \hline
\end{tabular}%
\end{equation*}%
$\allowbreak $

Some values of the $3$-associated $3$-Stirling numbers of the second kind

\begin{equation*}
\begin{tabular}{c|cccc}
\hline
$n\backslash k$ & $3$ & $4$ & $5$ & $6$ \\ \hline
$6$ & $\allowbreak 48$ &  &  &  \\ 
$7$ & $864$ &  &  &  \\ 
$8$ & $12\,240$ & $960$ &  &  \\ 
$9$ & $166\,320$ & $31\,680$ &  &  \\ 
$10$ & $2298\,240$ & $735\,840$ & $20\,160$ &  \\ 
$11$ & $33\,022\,080$ & $15\,200\,640$ & $1048\,320$ &  \\ 
$12$ & $497\,871\,360$ & $302\,279\,040$ & $35\,925\,120$ & $483\,840$ \\ 
$13$ & $7903\,526\,400$ & $5994\,777\,600$ & $1043\,280\,000$ & $%
36\,288\,000 $ \\ 
$14$ & $132\,\allowbreak 204\,441\,600$ & $120\,\allowbreak 708\,403\,200$ & 
$28\,\allowbreak 101\,427\,200$ & $1716\,422\,400$ \\ 
$15$ & $2328\,\allowbreak 905\,779\,200$ & $2491\,\allowbreak 766\,323\,200$
& $732\,\allowbreak 872\,448\,000$ & $66\,\allowbreak 501\,388\,800$ \\ 
$16$ & $43\,153\,\allowbreak 254\,144\,000$ & $53\,016\,\allowbreak
855\,091\,200$ & $18\,942\,\allowbreak 597\,273\,600$ & $2325\,\allowbreak
792\,268\,800$ \\ 
$17$ & $839\,788\,\allowbreak 479\,129\,600$ & $1166\,096\,\allowbreak
823\,091\,200$ & $491\,947\,\allowbreak 097\,241\,600$ & $%
77\,022\,\allowbreak 020\,275\,200$ \\ \hline
\end{tabular}%
\end{equation*}

Some values for the $2$-associated $3$-Lah numbers

\begin{equation*}
\begin{tabular}{c|cccc}
\hline
$n\backslash k$ & $2$ & $3$ & $4$ & $5$ \\ \hline
$6$ & $\allowbreak 216$ &  &  &  \\ 
$7$ & $2880$ &  &  &  \\ 
$8$ & $33\,120$ &  &  &  \\ 
$9$ & $383\,040$ & $45\,360$ &  &  \\ 
$10$ & $4636\,800$ & $1330\,560$ &  &  \\ 
$11$ & $59\,512\,320$ & $28\,667\,520$ &  &  \\ 
$12$ & $812\,851\,200$ & $562\,464\,000$ & $16\,329\,600$ &  \\ 
$13$ & $11\,\allowbreak 815\,372\,800$ & $10\,\allowbreak 777\,536\,000$ & $%
838\,252\,800$ &  \\ 
$14$ & $182\,\allowbreak 499\,609\,600$ & $207\,\allowbreak 886\,694\,400$ & 
$28\,\allowbreak 979\,596\,800$ &  \\ 
$15$ & $2988\,\allowbreak 969\,984\,000$ & $4097\,\allowbreak 379\,686\,400$
& $859\,\allowbreak 328\,870\,400$ & $\allowbreak 9340\,531\,200$ \\ 
$16$ & $51\,783\,\allowbreak 904\,972\,800$ & $83\,168\,\allowbreak
089\,804\,800$ & $23\,799\,\allowbreak 673\,497\,600$ & $741\,\allowbreak
015\,475\,200$ \\ 
$17$ & $946\,756\,\allowbreak 242\,432\,000$ & $1745\,745\,\allowbreak
281\,280\,000$ & $640\,760\,\allowbreak 440\,320\,000$ & $%
37\,486\,\allowbreak 665\,216\,000$ \\ \hline
\end{tabular}%
\end{equation*}%
Some values for the $3$-associated $2$-Lah numbers

\begin{equation*}
\begin{tabular}{c|cccc}
\hline
$n\backslash k$ & $3$ & $4$ & $5$ & $6$ \\ \hline
$9$ & $\allowbreak 19\,440$ &  &  &  \\ 
$10$ & $544\,320$ &  &  &  \\ 
$11$ & $11\,249\,280$ &  &  &  \\ 
$12$ & $212\,647\,680$ & $9797\,760$ &  &  \\ 
$13$ & $3940\,876\,800$ & $489\,888\,000$ &  &  \\ 
$14$ & $73\,\allowbreak 766\,246\,400$ & $16\,\allowbreak 525\,555\,200$ & 
&  \\ 
$15$ & $1414\,\allowbreak 970\,726\,400$ & $479\,\allowbreak 001\,600\,000$
& $6466\,521\,600$ &  \\ 
$16$ & $28\,021\,\allowbreak 593\,600\,000$ & $12\,989\,\allowbreak
565\,388\,800$ & $504\,\allowbreak 388\,684\,800$ &  \\ 
$17$ & $575\,115\,\allowbreak 187\,046\,400$ & $342\,959\,\allowbreak
397\,580\,800$ & $25\,107\,\allowbreak 347\,865\,600$ &  \\ 
$18$ & $12\,255\,524\,\allowbreak 176\,896\,000$ & $9007\,261\,\allowbreak
046\,784\,000$ & $\allowbreak 1030\,447\,\allowbreak 401\,984\,000$ & $%
5884\,\allowbreak 534\,656\,000$ \\ 
$19$ & $271\,347\,662\,\allowbreak 057\,472\,000$ & $238\,268\,731\,%
\allowbreak 244\,544\,000$ & $38\,309\,628\,\allowbreak 284\,928\,000$ & $%
659\,067\,\allowbreak 881\,472\,000$ \\ 
$20$ & $6242\,314\,363\,\allowbreak 084\,800\,000$ & $6397\,038\,394\,%
\allowbreak 306\,560\,000$ & $1350\,900\,851\,\allowbreak 908\,608\,000$ & $%
45\,350\,147\,\allowbreak 082\,240\,000$ \\ \hline
\end{tabular}%
\end{equation*}%
Some values for the $3$-associated $3$-Lah numbers
\end{center}


\begin{thebibliography}{99}
\bibitem{MR536963} J. C. Ahuja and E. A. Enneking. Concavity property and a recurrence relation for associated Lah numbers. Fibonacci Quart., 17(2):158--161, 1979.
\bibitem{MR2900003} C. B. Corcino, R. B. Corcino, and N. Acala. Asymptotic estimates for r-Whitney numbers of the second kind. Journal of Applied Mathematics, 2014(354053):7, 2014. \textbf{10}, Art. 07.2.3. (2007).
\bibitem{BelBel13} H. Belbachir and A. Belkhir. Cross recurrence relations for r-Lah numbers. Ars Combin., 110:199--203, 2013.
\bibitem{BelBou141} H. Belbachir, A. Belkhir, and I. E. Bousbaa. Combinatorial approach of the generalized Stirling numbers. Submitted.
\bibitem{BelBou12} H. Belbachir and I. E. Bousbaa. Convolution identities for the r-Stirling numbers. Submitted.
\bibitem{BelBou142} H. Belbachir and I. E. Bousbaa. A simple combinatorial interpretation of the Whiteny and r-Whitney numbers. Submitted.
\bibitem{MR25} H. Belbachir and I. E. Bousbaa. Translated Whitney and r-Whitney numbers: A combinatorial approach. Journal of Integer Sequences, 16:13.8.6, 2013.
\bibitem{BelBou14} H. Belbachir and I. E. Bousbaa. Combinatorial identities for the r-Lah numbers. Ars Combin., 117, 2014.
\bibitem{MR1415279} M. Benoumhani. On Whitney numbers of Dowling lattices. Discrete Math., 159(1-3):13--33, 1996.
\bibitem{MR1453407} M. Benoumhani. On some numbers related to Whitney numbers of Dowling lattices. Adv. in Appl. Math., 19(1):106--116, 1997.
\bibitem{MR1659426} M. Benoumhani. Log-concavity of Whitney numbers of Dowling lattices. Adv. in Appl. Math., 22(2):186--189, 1999.
\bibitem{MR743795} A. Z. Broder. The r-Stirling numbers. Discrete Math., 49(3):241--259, 1984.
\bibitem{MR531621} L. Carlitz. Degenerate Stirling, Bernoulli and Eulerian numbers. Utilitas Math., 15:51--88, 1979.
\bibitem{MR570168} L. Carlitz. Weighted Stirling numbers of the first and second kind, I. Fibonacci Quart., 18(2):147--162, 1980.
\bibitem{MR599657}L. Carlitz. Weighted Stirling numbers of the first and second kind, II. Fibonacci Quart., 18(3):242--257, 1980.
\bibitem{MR2926106} G.-S. Cheon and J.-H. Jung. r-Whitney numbers of Dowling lattices. Discrete Math., 312(15):2337--2348, 2012.
\bibitem{MR1999993} T. Comtet. Advanced Combinatorics. D. Reidel, Boston, DC, 1974.
\bibitem{MR0354401} H. W. Gould. Combinatorial identities. Henry W. Gould, Morgantown, W.Va., 1972. A standardized set of tables listing 500 binomial coeficient summations.
\bibitem{MR1397498} R. L. Graham, D. E. Knuth, and O. Patashnik. Concrete mathematics. Addison-Wesley Publishing Company, Reading, MA, second edition, 1994. A foundation for computer science.
\bibitem{MR600368} F. T. Howard. Associated Stirling numbers. Fibonacci Quart., 18(4):303--315, 1980.
\bibitem{MR742846} F. T. Howard. Weighted associated Stirling numbers. Fibonacci Quart., 22(2):156--165, 1984.
\bibitem{MR1618435} L. C. Hsu and P. J.-S. Shiue. A unified approach to generalized Stirling numbers. Adv. in Appl. Math., 20(3):366--384, 1998.
\bibitem{MR1949650} J. Riordan. An introduction to combinatorial analysis. Dover Publications Inc., Mineola, NY, 2002.

\bibitem{MR1992789} N. J. A. Sloane. The on-line encyclopedia of integer sequences. Notices Amer. Math. Soc., 50(8):912--915, 2003.
\end{thebibliography}
\end{document}